\documentclass[12pt,twoside]{amsart}
\usepackage{amssymb,amsmath,amsthm, amscd, enumerate, mathrsfs}
\usepackage{graphicx, hhline}
\usepackage[all]{xy}
\usepackage[usenames]{color}
\usepackage{hyperref}
\hypersetup{colorlinks=true}

\title[Minimal model program]{Some remarks on the minimal model program for 
log canonical pairs}
\author{Osamu Fujino} 
\date{2014/6/25, version 0.41}
\address{Department of Mathematics, Graduate School of Science, 
Kyoto University, Kyoto 606-8502, Japan}
\email{fujino@math.kyoto-u.ac.jp}

\newcommand{\Pic}[0]{{\operatorname{Pic}}}
\newcommand{\mult}[0]{{\operatorname{mult}}}
\newcommand{\rank}[0]{{\operatorname{rank}}}
\newcommand{\Proj}[0]{{\operatorname{Proj}}}
\newcommand{\Exc}[0]{{\operatorname{Exc}}}
\newcommand{\Supp}[0]{{\operatorname{Supp}}}
\newcommand{\Bs}[0]{{\operatorname{Bs}}}

\newtheorem{thm}{Theorem}[section]
\newtheorem{lem}[thm]{Lemma}
\newtheorem{cor}[thm]{Corollary}

\newtheorem{conj}[thm]{Conjecture}
\newtheorem{cla}{Claim}

\theoremstyle{definition}
\newtheorem{ex}[thm]{Example}
\newtheorem{defn}[thm]{Definition}
\newtheorem{rem}[thm]{Remark}
\newtheorem*{ack}{Acknowledgments} 
\newtheorem{say}[thm]{}
\newtheorem{case}{Case}
\begin{document}

\maketitle 

\begin{abstract}
We prove that the target space of an extremal Fano contraction from 
a log canonical pair has only log canonical singularities. 
We also treat some related topics, for example, the finite generation 
of canonical rings for compact K\"ahler manifolds, and so on. 
The main ingredient of this paper is the nefness of 
the moduli parts of lc-trivial fibrations. 
We also give some observations on the semi-ampleness of the moduli 
parts of lc-trivial fibrations. 
For the reader's convenience, 
we discuss some examples of non-K\"ahler manifolds, flopping 
contractions, and so on, in order to clarify our results. 
\end{abstract}

\tableofcontents 

\section{Introduction}\label{sec1}

Let $\pi:(X, \Delta)\to S$ be a projective morphism from a log canonical 
pair $(X, \Delta)$ to a variety $S$. 
Then the cone theorem 
$$
\overline {NE}(X/S)=\overline {NE}(X/S)_{K_X+\Delta\geq 0}+\sum _{i}
\mathbb R_{\geq 0}[C_i]
$$ 
holds for $\pi:(X, \Delta)\to S$. We take a $(K_X+\Delta)$-negative extremal 
ray $R=\mathbb R_{\geq 0}[C_i]$. Then 
there is a contraction morphism 
$$
\varphi_R: (X, \Delta)\to Y
$$ 
over $S$ associated to $R$. 
For the details of the cone and contraction theorem for log canonical 
pairs, see \cite{ambro1}, 
\cite{fujino3}, \cite{fujino4}, \cite{fujino5}, and \cite[Theorem 1.1]{fujino6} 
(see also \cite{fujino-foundation}).

From now on, let us consider a contraction morphism 
$$
f:(X, \Delta)\to Y
$$ 
such that 
\begin{itemize}
\item[(i)] $(X, \Delta)$ is a $\mathbb Q$-factorial log canonical 
pair, 
\item[(ii)] $-(K_X+\Delta)$ is $f$-ample, and 
\item[(iii)] $\rho(X/Y)=1$. 
\end{itemize}
Then we have the following three cases. 

\begin{case}[Divisorial contraction] 
$f$ is divisorial, that is, $f$ is a birational contraction which contracts 
a divisor. 

In this case, the exceptional locus $\Exc(f)$ of $f$ is a prime divisor 
on $X$ and $(Y, \Delta_Y)$ is a $\mathbb Q$-factorial log canonical pair with 
$\Delta_Y=f_*\Delta$. 
\end{case}

\begin{case}[Flipping contraction]\label{case2} 
$f$ is flipping, that is, $f$ is a birational contraction which is small. 

In this case, we can take the flipping diagram: 
$$
\xymatrix{
X\ar@{-->}[rr]^{\varphi} \ar[rd]_{f} && X^+\ar[ld]^{f^+} \\
  & Y & 
} 
$$
where $f^+$ is a small projective birational morphism and  
\begin{itemize}
\item[(i$'$)] $(X^+, \Delta^+)$ is a $\mathbb Q$-factorial 
log canonical pair with $\Delta^+=\varphi_*\Delta$, 
\item[(ii$'$)] $K_{X^+}+\Delta^+$ is $f^+$-ample, and 
\item[(iii$'$)] $\rho(X^+/Y)=1$. 
\end{itemize}
For the existence of log canonical flips, see \cite[Corollary 1.2]{birkar} and 
\cite[Corollary 1.8]{hacon-xu}. 
\end{case}

\begin{case}[Fano contraction]\label{case3}  
$f$ is a Fano contraction, that is, $\dim Y<\dim X$. 

Then $Y$ is $\mathbb Q$-factorial and has only log canonical singularities. 
Moreover, if every log canonical center of $(X, \Delta)$ is dominant onto $Y$, 
then $Y$ has only log terminal singularities. 

In Case \ref{case3}, $f:(X, \Delta)\to Y$ is usually called 
a {\em{Mori fiber space}}. 
\end{case} 

The log canonicity of $Y$ in Case \ref{case3} is missing 
in the literature. So we prove it in this paper. 
It is an easy consequence of the following theorem. 
For the other statements on singularities 
in the above three cases, see, for example, 
\cite[Propositions 3.36, 3.37, Corollaries 3.42, and 3.43]{kollar-mori} 
(see also \cite{fujino-foundation}). 

\begin{thm}[{cf.~\cite[Theorem 1.2]{fujino1}}]\label{thm11}
Let $(X, \Delta)$ be a sub log canonical pair such that 
$X$ is smooth and $\Supp \Delta$ is a simple normal crossing divisor on $X$. 
Let $f:(X, \Delta)\to Y$ be 
a proper surjective morphism 
such that 
$$f_*\mathcal O_X(\lceil -\Delta^{<1}\rceil)\simeq \mathcal O_Y$$ and that 
$$K_X+\Delta
\sim _{\mathbb Q, f}0. $$ 
Assume that $K_Y$ is $\mathbb Q$-Cartier. Then 
$Y$ has only log canonical singularities. 
We further assume that every log canonical center of $(X, \Delta)$ is dominant 
onto $Y$. Then $Y$ has only log terminal singularities. 
\end{thm}

Our proof of Theorem \ref{thm11} depends on the nefness of 
the moduli parts of lc-trivial fibrations (cf.~\cite[Section 5, Part II]{mori}, 
\cite{kawamata2}, \cite{ambro2}, \cite{fujino-un}, 
\cite{kollar2}, \cite[Section 3]{fujino-gongyo3}, and 
so on). 
In this paper, we use Ambro's formulation in \cite{ambro2} and 
its generalization in \cite[Section 3]{fujino-gongyo3} based on the semipositivity 
theorem in \cite{fujino2}. For the details of the 
Hodge theoretic aspects of the semipositivity theorem, 
see also \cite{fujino-fujisawa} and \cite{ffs}. 
It is conjectured that the moduli parts of lc-trivial fibrations 
are semi-ample (see Conjecture \ref{conj29}). We give some observations on the semi-ampleness 
of the moduli parts of lc-trivial fibrations in Section \ref{sec-new}. 

By the proof of \cite[Theorem 1.2]{fujino1} and 
\cite[Section 3]{fujino-gongyo3} (see Theorem \ref{thm25}), we have: 

\begin{thm}[{cf.~\cite[Theorem 1.2]{fujino1} and 
\cite[Theorem 4.2.1]{fujino-un}}]\label{thm-add} 
Let $(X, \Delta)$ be a sub log canonical pair 
such that $X$ is smooth and $\Supp \Delta$ is a simple normal crossing divisor 
on $X$. 
Let $f:(X, \Delta)\to Y$ be a proper surjective morphism such that 
$$
f_*\mathcal O_X(\lceil -\Delta^{<1}\rceil)\simeq \mathcal O_Y
$$ 
and that 
$$
K_X+\Delta\sim _{\mathbb Q}f^*D
$$ 
for some $\mathbb Q$-Cartier $\mathbb Q$-divisor 
$D$ on $Y$. 
Assume that $\pi:Y\to S$ is a projective morphism onto a quasi-projective 
variety $S$. 
Let $A$ be a $\pi$-ample Cartier divisor on $Y$ and let $\varepsilon$ be an arbitrary 
positive rational number. 
We further assume that every log canonical 
center of $(X, \Delta)$ is dominant onto $Y$. 
Then there is an effective $\mathbb Q$-divisor 
$\Delta_Y$ on $Y$ such that 
$$
K_Y+\Delta_Y\sim _{\mathbb Q, \pi}D+\varepsilon A 
$$ 
and that $(Y, \Delta_Y)$ is kawamata log terminal. 
\end{thm} 

Let us recall some results in \cite{reid}, \cite{kollar1}, and \cite{fujino1} 
for the reader's convenience. 

\begin{rem}[Known results]
Let $f:X\to Y$ be a contraction morphism associated to 
a $K_X$-negative extremal face such that $X$ is a projective 
variety with only canonical singularities. 
Then it is well known that $Y$ has only rational singularities by 
\cite[Corollary 7.4]{kollar1}. 
It was first proved by Reid when $\dim X\leq 3$ (see \cite{reid}). 

Let $f:(X, \Delta)\to Y$ be a contraction morphism 
associated to a $(K_X+\Delta)$-negative extremal face 
such that $(X, \Delta)$ is a projective divisorial log terminal pair. 
Then there is an effective $\mathbb Q$-divisor 
$\Delta_Y$ on $Y$ such that $(Y, \Delta_Y)$ is kawamata log terminal 
by \cite[Corollary 4.5]{fujino1}. 
In particular, $Y$ has only rational singularities. 

Note that the above results now easily follow from Theorem \ref{thm-add}. 
\end{rem}

The following conjecture is related to Theorem \ref{thm11} 
(cf.~\cite[Conjecture 7.4]{kawamata1}). 

\begin{conj}\label{conj12}Let $(X, \Delta)$ be a projective log canonical 
pair. Assume that 
the log canonical ring 
$$
R(X, \Delta)=\bigoplus _{m \geq 0}H^0(X, \mathcal O_X(\lfloor m(K_X+\Delta)\rfloor))
$$ 
is a finitely generated $\mathbb C$-algebra. 
We put 
$$
Y=\Proj R(X, \Delta). 
$$
Then there is an effective $\mathbb Q$-divisor 
$\Delta_Y$ on $Y$ such that $(Y, \Delta_Y)$ is log canonical. 
\end{conj}

If $(X, \Delta)$ is kawamata log terminal in Conjecture \ref{conj12}, 
then we have: 

\begin{thm}\label{thm14} 
Let $(X, \Delta)$ be a projective kawamata log terminal pair such that 
$\Delta$ is a $\mathbb Q$-divisor. 
We put 
$$
Y=\Proj R(X, \Delta). 
$$ 
Then there is an effective $\mathbb Q$-divisor $\Delta_Y$ 
on $Y$ such that 
$(Y, \Delta_Y)$ is kawamata log terminal. 
\end{thm}

It is a generalization of Nakayama's result (see \cite[Theorem]{nakayama}), 
which is a complete solution of \cite[Conjecture 7.4]{kawamata1}. 
Theorem \ref{thm13} is a partial answer to Conjecture \ref{conj12}. 

\begin{thm}\label{thm13}
Let $(X, \Delta)$ be a projective log canonical pair 
such that the log canonical ring $R(X, \Delta)$ is a finitely generated 
$\mathbb C$-algebra. 
We assume that $K_Y$ is $\mathbb Q$-Cartier where 
$Y=\Proj R(X, \Delta)$. 
Then $Y$ has only log canonical singularities. 
\end{thm}

Note that $K_Y$ is not always $\mathbb Q$-Cartier in Conjecture \ref{conj12}. 
Therefore, Theorem \ref{thm13} is far from a complete solution of 
Conjecture \ref{conj12}. 

The following conjecture is open. It is closely related to 
Conjecture \ref{conj12} and Theorem \ref{thm11}. 

\begin{conj}\label{conj15}
Let $(X, \Delta)$ be a projective log canonical pair 
and let $f:X\to Y$ be a contraction morphism between normal projective 
varieties such that 
$$
K_X+\Delta\sim _{\mathbb R, f}0. 
$$ 
Then there is an effective $\mathbb R$-divisor 
$\Delta_Y$ on $Y$ such that $(Y, \Delta_Y)$ is log canonical 
and 
$$
K_X+\Delta\sim _{\mathbb R} f^*(K_Y+\Delta_Y). 
$$
\end{conj}

Of course, Conjecture \ref{conj15} follows from 
the b-semi-ampleness conjecture of the moduli parts of 
lc-trivial fibrations (see Conjecture \ref{conj29} and 
Remark \ref{rem24}). 

From now on, the variety $X$ is not always algebraic. 
We treat compact K\"ahler manifolds. 
The following theorem is also missing in the literature. 
Note that we can reduce the problem to the case when the variety is projective 
by taking the Iitaka fibration. When $X$ is projective, Theorem \ref{thm18} 
is well known (see \cite{bchm}). 

\begin{thm}[{cf.~\cite{bchm} and \cite{fujino-mori}}]\label{thm18}  
Let $X$ be a compact K\"ahler manifold and 
let $\Delta$ be an effective $\mathbb Q$-divisor 
on $X$ such that $(X, \Delta)$ is kawamata log terminal. 
Then the log canonical ring 
$$
R(X, \Delta)=\bigoplus _{m\geq 0}H^0(X, \mathcal O_X(\lfloor m(K_X+\Delta)\rfloor))
$$ 
is a finitely generated $\mathbb C$-algebra. 
\end{thm}

As a special case of Theorem \ref{thm18}, we have: 

\begin{cor}\label{cor19} 
Let $X$ be a compact K\"ahler manifold. 
Then the canonical ring 
$$
R(X)=\bigoplus _{m \geq 0} H^0(X, \omega_X^{\otimes m})
$$ 
is a finitely generated $\mathbb C$-algebra. 
\end{cor}

We note that there exists a compact complex non-K\"ahler 
manifold whose canonical ring 
is not a finitely generated $\mathbb C$-algebra (see 
Example \ref{ex-wilson}). 

The following conjecture is still open even when 
$X$ is projective. 

\begin{conj}\label{conj20}
Let $X$ be a compact K\"ahler manifold and let $\Delta$ be 
an effective $\mathbb Q$-divisor on $X$ such that 
$(X, \Delta)$ is log canonical. 
Then 
the log canonical 
ring 
$$
R(X, \Delta)=\bigoplus _{m\geq 0}H^0(X, \mathcal O_X(\lfloor m(K_X+\Delta)\rfloor))
$$ 
is a finitely generated $\mathbb C$-algebra. 
\end{conj}

We do not know if we can reduce Conjecture \ref{conj20} to the case 
when the variety is projective or not (see Remark \ref{rem47}). 

From Section \ref{sec-pre} to Section \ref{sec2}, we assume that 
all the varieties are algebraic for simplicity, although some of 
the results can be generalized to analytic varieties. 
Section \ref{sec-pre} collects some basic definitions. 
In Section \ref{sec-new}, we discuss lc-trivial fibrations 
and give some new observations. 
Section \ref{sec2} is devoted to the proofs of the main results. 
In Section \ref{sec3}, we discuss some analytic generalizations 
and related topics. 
We note that we just explain how to adapt the arguments to 
the analytic settings and discuss Theorem \ref{thm18}, 
Corollary \ref{cor19}, and so on. 
In Section \ref{sec-counterexamples}, we discuss some examples 
of non-K\"ahler manifolds, which clarify 
the main difference between K\"ahler manifolds and 
non-K\"ahler manifolds. Note that 
Corollary \ref{cor19} can not be generalized for non-K\"ahler 
manifolds (see Example \ref{ex-wilson}). 
In Section \ref{sec-appendix}:~Appendix, 
we quickly discuss the minimal model program for log canonical 
pairs and describe some related examples by J\'anos Koll\'ar 
for the reader's convenience. 

\begin{ack} 
The author was partially supported by the Grant-in-Aid for Young Scientists 
(A) $\sharp$24684002 from JSPS. 
He would like to thank Professor Shigefumi Mori and Yoshinori Gongyo for 
useful comments and questions. 
\end{ack}

This paper is a supplement to the author's previous 
papers \cite{fujino1}, \cite{fujino6}, and so on. 
For some recent related topics, see, for example, \cite{alexeev-borisov}, 
\cite{birkar2}, \cite{chp}, \cite{horing-peternell}, and \cite{horing-peternell2}. 

We will work over $\mathbb C$, the complex number field, throughout 
this paper. We will make use of the standard notation as in \cite{kollar-mori} 
and \cite{fujino6}. 

\section{Preliminaries}\label{sec-pre} 

Let us recall some basic definitions on 
singularities of pairs. For the details, see \cite{kollar-mori} 
and \cite{fujino6}. 

\begin{say}[Pairs] 
A pair $(X, \Delta)$ consists of a normal variety $X$ and 
an $\mathbb R$-divisor $\Delta$ on $X$ such that 
$K_X+\Delta$ is $\mathbb R$-Cartier. 
A pair $(X, \Delta)$ is called {\em{sub kawamata log terminal}} 
(resp.~{\em{sub log canonical}}) if for any 
proper birational morphism 
$g:Z\to X$ from a normal variety $Z$, every coefficient 
of $\Delta_Z$ is $<1$ (resp.~$\leq 1$) where 
$$K_Z+\Delta_Z:=g^*(K_X+\Delta).$$ 
A pair 
$(X, \Delta)$ is called {\em{kawamata log terminal}} 
(resp.~{\em{log canonical}}) if $(X, \Delta)$ is sub kawamata log terminal 
(resp.~sub log canonical) and $\Delta$ is effective. 
If $(X, 0)$ is kawamata log terminal, then we simply say that 
$X$ has only {\em{log terminal singularities}}. 

Let $(X, \Delta)$ be a sub log canonical pair and let $W$ be a closed 
subset of $X$. 
Then $W$ is called a {\em{log canonical center}} of $(X, \Delta)$ if there are 
a proper birational morphism 
$g:Z\to X$ from a normal variety $Z$ and a prime divisor 
$E$ on $Z$ such that $\mult _E\Delta_Z=1$ and $g(E)=W$. 

We note that $-\mult _E\Delta_Z$ is denoted by $a(E, X, \Delta)$ 
and 
is called the {\em{discrepancy coefficient}} of $E$ with respect to $(X, \Delta)$. 

Let $D=\sum _id_iD_i$ be an $\mathbb R$-divisor on $X$ such that 
$D_i$ is a prime divisor for every $i$ and that $D_i\ne D_j$ for $i\ne j$. 
Then $\lceil D\rceil$ (resp.~$\lfloor D\rfloor$) denotes 
the {\em{round-up}} (resp.~{\em{round-down}}) of 
$D$. 
We put 
$$
D^{<1}=\sum _{d_i<1}d_iD_i. 
$$
\end{say}

In this paper, we use the notion of {\em{b-divisors}} introduced by 
Shokurov. 
For the details, see, for example, \cite[Section 3]{fujino6.5}. 

\begin{say}[Canonical b-divisors and discrepancy b-divisors]
Let $X$ be a normal variety and let $\omega$ be a top rational differential 
form of $X$. 
Then $(\omega)$ defines a b-divisor $\mathbf K$. 
We call $\mathbf K$ the {\em{canonical b-divisor}} of $X$. 
The {\em{discrepancy b-divisor}} $\mathbf A=\mathbf A(X, \Delta)$ 
of a pair $(X, \Delta)$ is the $\mathbb R$-b-divisor 
of $X$ with the trace $\mathbf A_Y$ defined by the 
formula 
$$
K_Y=f^*(K_X+\Delta)+\mathbf A_Y, 
$$ 
where $f:Y\to X$ is a proper birational morphism of normal varieties. 
Similarly, we define $\mathbf A^*=\mathbf A^*(X, \Delta)$ by 
$$
\mathbf A_Y^*=\sum _{a_i>-1}a_i A_i
$$ 
for 
$$
K_Y=f^*(K_X+\Delta)+\sum a_i A_i, 
$$ 
where $f:Y\to X$ is a proper birational morphism of normal varieties. 
\end{say}

\begin{say}[b-nef and b-semi-ample $\mathbb Q$-b-divisors]\label{say2}  
Let $X$ be a normal variety and let $X\to S$ be a proper surjective morphism onto a variety $S$. 
A $\mathbb Q$-b-divisor $\mathbf D$ of $X$ is {\em{b-nef over $S$}} 
(resp.~{\em{b-semi-ample over $S$}}) if there exists a proper birational morphism $X'\to X$ from a normal 
variety $X'$ such that $\mathbf D=\overline {\mathbf D_{X'}}$ and $\mathbf D_{X'}$ is nef 
(resp.~semi-ample) relative to the induced morphism $X'\to S$. 
A $\mathbb Q$-b-divisor $\mathbf D$ of $X$ is {\em{$\mathbb Q$-b-Cartier}} 
if there is a proper birational morphism $X'\to X$ from a normal 
variety $X'$ such that $\mathbf D=\overline{\mathbf D_{X'}}$. 
\end{say} 

\section{Lc-trivial fibrations}\label{sec-new} 

Let us recall the definition of {\em{lc-trivial fibrations}}. 

\begin{defn}[Lc-trivial fibrations]\label{def22} 
An {\em{lc-trivial fibration}} $f:(X, \Delta)\to Y$ consists of 
a proper surjective morphism between normal varieties with 
connected fibers and a pair $(X, \Delta)$ satisfying the 
following properties: 
\begin{itemize}
\item[(i)] $(X, \Delta)$ is sub log canonical over the 
generic point of $Y$, 
\item[(ii)] $\rank f_*\mathcal O_X(\lceil \mathbf A^*(X, \Delta)\rceil)=1$, and 
\item[(iii)] there exists a $\mathbb Q$-Cartier $\mathbb Q$-divisor 
$D$ on $Y$ such that 
$$
K_X+\Delta\sim _{\mathbb Q}f^*D. 
$$
\end{itemize}
\end{defn}

\begin{rem}
Let $f:X\to Y$ be a proper surjective morphism 
between normal varieties with $f_*\mathcal O_X\simeq \mathcal O_Y$. 
Assume that $(X, \Delta)$ is log canonical over the generic point of 
$Y$. 
Then we have 
$$
\rank f_*\mathcal O_X(\lceil \mathbf A^*(X, \Delta)\rceil)=1. 
$$
\end{rem}

We give a standard example of lc-trivial fibrations.

\begin{ex}
Let $(X, \Delta)$ be a sub log canonical pair such that 
$X$ is smooth and that $\Supp \Delta$ is a simple normal crossing 
divisor on $X$. Let $f:X\to Y$ be a proper surjective 
morphism onto a normal variety $Y$ such that 
$$
K_X+\Delta\sim _{\mathbb Q, f}0
$$ 
and that 
$$
f_*\mathcal O_X(\lceil -\Delta^{<1}\rceil)\simeq \mathcal O_Y. 
$$ 
Then $f:(X, \Delta)\to Y$ is an lc-trivial fibration.
\end{ex}

We give a remark on the definition of lc-trivial fibrations 
for the reader's convenience. 

\begin{rem}[Lc-trivial fibrations and klt-trivial fibrations]\label{rem25}
In \cite[Definition 2.1]{ambro2}, $(X, \Delta)$ is assumed to 
be sub kawamata log terminal over the generic point of $Y$. 
Therefore, Definition \ref{def22} is wider than Ambro's original 
definition of lc-trivial fibrations. 
When $(X, \Delta)$ is sub kawamata log terminal over the 
generic point of $Y$ in Definition \ref{def22}, we call 
$f:(X, \Delta)\to Y$ a {\em{klt-trivial fibration}} (see \cite[Definition 3.1]{fujino-gongyo3}). 
\end{rem}

We need the notion of {\em{induced lc-trivial fibrations}}, 
{\em{discriminant $\mathbb Q$-divisors}}, {\em{moduli $\mathbb Q$-divisors}}, 
and so on in order to discuss lc-trivial fibrations. 

\begin{say}[Induced lc-trivial fibrations by base changes]\label{say23}
Let $f:(X, \Delta)\to Y$ be an lc-trivial 
fibration and let $\sigma:Y'\to Y$ be a generically finite 
morphism. Then we have an induced lc-trivial fibration $f':(X', \Delta_{X'})\to Y'$, 
where 
$\Delta_{X'}$ is defined by $\mu^*(K_X+\Delta)=K_{X'}+\Delta_{X'}$: 
$$
\xymatrix{
   (X', \Delta_{X'}) \ar[r]^{\mu} \ar[d]_{f'} & (X, \Delta)\ar[d]^{f} \\
   Y' \ar[r]_{\sigma} & Y,
} 
$$
Note that $X'$ is the normalization of 
the main component of $X\times _{Y}Y'$. 
We sometimes replace $X'$ with $X''$ where $X''$ is a normal variety 
such that there is a proper birational morphism 
$\varphi:X''\to X'$. 
In this case, we set 
$K_{X''}+\Delta_{X''}=\varphi^*(K_{X'}+\Delta_{X'})$.  
\end{say}

\begin{say}
[Discriminant $\mathbb Q$-b-divisors and moduli $\mathbb Q$-b-divisors]\label{say28} 
Let us consider an lc-trivial fibration 
$f:(X, \Delta)\to Y$ as in Definition \ref{def22}. 
We take a prime divisor $P$ on $Y$. 
By shrinking $Y$ around the generic point of $P$, 
we assume that $P$ is Cartier. We set 
$$
b_P=\max \left\{t \in \mathbb Q\, \left|\, 
\begin{array}{l}  {\text{$(X, \Delta+tf^*P)$ is sub log canonical}}\\
{\text{over the generic point of $P$}} 
\end{array}\right. \right\} 
$$ 
and 
set $$
B_Y=\sum _P (1-b_P)P, 
$$ 
where $P$ runs over prime divisors on $Y$. Then it is easy to  see that 
$B_Y$ is a well-defined $\mathbb Q$-divisor on $Y$ and is called the {\em{discriminant 
$\mathbb Q$-divisor}} of $f:(X, \Delta)\to Y$. We set 
$$
M_Y=D-K_Y-B_Y
$$ 
and call $M_Y$ the {\em{moduli $\mathbb Q$-divisor}} of $f:(X, \Delta)\to Y$. 
Let $\sigma:Y'\to Y$ be a proper birational morphism 
from a normal variety $Y'$ and let $f':(X', \Delta_{X'})\to Y'$ be the 
induced lc-trivial fibration 
by $\sigma:Y'\to Y$ (see \ref{say23}). 
We can define $B_{Y'}$, $K_{Y'}$ and $M_{Y'}$ such that 
$$\sigma^*D=K_{Y'}+B_{Y'}+M_{Y'},$$ 
$\sigma_*B_{Y'}=B_Y$, $\sigma _*K_{Y'}=K_Y$, and $\sigma_*M_{Y'}=M_Y$. Hence 
there exist a unique $\mathbb Q$-b-divisor $\mathbf B$ such that 
$\mathbf B_{Y'}=B_{Y'}$ for every $\sigma:Y'\to Y$ and a unique 
$\mathbb Q$-b-divisor $\mathbf M$ such that $\mathbf M_{Y'}=M_{Y'}$ for 
every $\sigma:Y'\to Y$. 
Note that $\mathbf B$ is called the {\em{discriminant $\mathbb Q$-b-divisor}} and 
that $\mathbf M$ is called the {\em{moduli $\mathbb Q$-b-divisor}} associated to 
$f:(X, \Delta)\to Y$. 
We sometimes simply say that $\mathbf M$ is the {\em{moduli part}} 
of $f:(X, \Delta)\to Y$. 
\end{say}

Theorem \ref{thm25} is the most fundamental 
result on lc-trivial fibrations. 
It is the main ingredient of this paper.  
Ambro \cite{ambro2} obtained Theorem \ref{thm25} 
for klt-trivial fibrations. Theorem \ref{thm25} is a direct generalization of 
\cite[Theorem 0.2]{ambro2}. 

\begin{thm}[{\cite[Theorem 3.6]{fujino-gongyo3}}]\label{thm25} 
Let $f:(X, \Delta)\to Y$ be an lc-trivial fibration and let $\pi:Y\to S$ 
be a proper morphism. 
Let $\mathbf B$ and $\mathbf M$ be the induced 
discriminant and moduli $\mathbb Q$-b-divisors of $f$. 
Then, 
\begin{itemize}
\item[(1)] $\mathbf K+\mathbf B$ is $\mathbb Q$-b-Cartier, 
\item[(2)] $\mathbf M$ is b-nef over $S$.  
\end{itemize}
\end{thm} 

\begin{rem}
Theorem \ref{thm25} says that 
there is a proper birational morphism $Y'\to Y$ from a normal variety $Y'$ such that 
$\mathbf K+\mathbf B=\overline {K_{Y'}+B_{Y'}}$, 
$\mathbf M=\overline {M_{Y'}}$, and 
$M_{Y'}$ is nef over $S$. 
We note that the arguments in \cite[Section 5]{ambro2} 
show how to construct $Y'$. 
For the details, see \cite[p.~245, set-up]{ambro2} 
and \cite[Proof of Theorem 2.7]{ambro2}. 
\end{rem}

The following conjecture is one of the most important 
open problems on lc-trivial fibrations. 
It was conjectured by Fujita, Mori, Shokurov and others (see \cite[Problem]{nakayama}, 
\cite[Conjecture 7.13]{prokhorov-shokurov} and so on). 

\begin{conj}[b-semi-ampleness conjecture]\label{conj29} 
Let $f:(X, \Delta)\to Y$ be an lc-trivial fibration and 
let $\pi:Y\to S$ be a proper morphism. 
Then the moduli part $\mathbf M$ is b-semi-ample over $S$. 
\end{conj}

Conjecture \ref{conj29} was only solved for some special cases 
(see \cite{kawamata-sub}, \cite{fujino-nagoya}, and \cite[Section 8]{prokhorov-shokurov}). 
The arguments in \cite{kawamata-sub} (see also 
\cite{prokhorov-shokurov}) and \cite{fujino-nagoya} 
use the theory of moduli spaces of curves, $K3$ surfaces, and Abelian 
varieties. 

We give some observations on Conjecture \ref{conj29}. 

\begin{say}[Observation I]\label{213}
Let $f:(X, \Delta)\to Y$ be an lc-trivial fibration. 
For simplicity, we assume that $X$ is smooth and $\Supp \Delta$ 
is a simple normal crossing divisor on $X$. 
We write 
$$
\Delta=\Delta^+-\Delta^-
$$ 
where $\Delta^+$ and $\Delta^-$ are effective 
$\mathbb Q$-divisors on $X$ such that 
$\Delta^+$ and $\Delta^-$ have no common irreducible components. 
In this situation, we have 
$$
\mathcal O_X(\lceil \mathbf A^*(X, \Delta)\rceil)\simeq \mathcal O_X(\lceil \Delta^-
\rceil)  
$$ 
over the generic point of $Y$ (see \cite[Lemma 3.22]{fujino6.5}). 
Therefore, the condition 
$$
\rank f_*\mathcal O_X(\lceil \mathbf A^*(X, \Delta)\rceil)=1
$$ 
is equivalent to 
$$
\rank f_*\mathcal O_X(\lceil \Delta^-\rceil)=1. 
$$ 
For Conjecture \ref{conj29}, it seems to be reasonable 
to assume that 
$$
\rank f_*\mathcal O_X(\lceil m\Delta^-\rceil)=1
$$ 
holds for every nonnegative integer $m$. 
This condition is equivalent to 
$$
\kappa (X_\eta, K_{X_\eta}+\Delta^+|_{X_{\eta}})=0
$$ 
where 
$X_\eta$ is the generic fiber of $f:X\to Y$. 
The condition 
$$
\rank f_*\mathcal O_X(\lceil \Delta^-\rceil)=1
$$ 
seems to be insufficient for Conjecture \ref{conj29}. 

If there are an lc-trivial fibration $f^\dag:(X^\dag, \Delta^\dag)\to Y$ such that 
$(X^\dag, \Delta^\dag)$ is log canonical and a proper birational morphism 
$\mu:X\to X^\dag$ such that 
$K_X+\Delta=\mu^*(K_{X^{\dag}}+\Delta^{\dag})$ and $f=f^\dag\circ \mu$, 
then $\Delta^-$ is $\mu$-exceptional. 
Therefore we have 
$$
\mu_*\mathcal O_X(\lceil m\Delta^-\rceil)\simeq \mathcal O_{X^\dag}
$$
for every nonnegative integer $m$. 
This implies 
$$
f_*\mathcal O_X(\lceil m\Delta^-\rceil)\simeq \mathcal O_Y
$$ 
for every nonnegative integer $m$. 
Consequently, this extra assumption that 
$$
\rank f_*\mathcal O_X(\lceil m\Delta^-\rceil)=1
$$ 
for every nonnegative integer $m$ is harmless for many applications. 
\end{say}

\begin{say}[Observation II]\label{214} 
Assume that the minimal model program and the abundance conjecture hold. 

Let $f:(X, \Delta)\to Y$ be an lc-trivial fibration such that $X$ is smooth and 
$\Supp \Delta$ is a simple normal 
crossing divisor on $X$. 
Assume that 
$$
\kappa (X_\eta, K_{X_\eta}+\Delta^+|_{X_\eta})=0
$$ 
as in \ref{213}. 
By \cite{abramovich-karu}, we can construct the following commutative diagram: 
$$
\xymatrix{
   X \ar[d]_{f} & X'\ar[l]_{\mu}\ar[d]^{f'} &\ar@{_{(}->}[l]U_{X'}\ar[d]\\
   Y  & \ar[l]^{\sigma}Y'& \ar@{_{(}->}[l] U_{Y'},
} 
$$
satisfying: 
\begin{itemize}
\item[(1)] $\mu$ and $\sigma$ are projective birational morphisms. 
\item[(2)] $f':(U_{X'}\subset X')\to (U_{Y'}\subset Y')$ is a projective equidimensional toroidal 
morphism. 
\item[(3)] $K_{X'}+\Delta_{X'}=\mu^*(K_X+\Delta)$, $\Supp \Delta_{X'}\subset \Sigma_{X'}
=X'\setminus U_{X'}$, and $X'$ is $\mathbb Q$-factorial. 
\item[(4)] $Y'$ is a smooth quasi-projective variety and $\Sigma_{Y'}=Y'\setminus U_{Y'}$ 
is a simple normal crossing divisor on $Y'$. 
\item[(5)] $f'$ is smooth over $U_{Y'}$ and $\Sigma_{X'}$ is a relatively normal crossing divisor over $U_{Y'}$. 
\end{itemize}
We can write 
$$
K_{X'}+\Delta_{X'}\sim _{\mathbb Q}f'^*(K_{Y'}+B_{Y'}+M_{Y'}) 
$$ 
as in \ref{say28}. 
Let $\Sigma_{Y'}=\sum _i P_i$ be the irreducible decomposition. 
Then we can write 
$$
B_{Y'}=\sum _i (1-b_{P_i})P_i
$$ 
as in \ref{say28}. 
We put 
$$
\Delta_{X'}+\sum _i b_{P_i}f'^*P_i=\Theta -E
$$ 
where $\Theta$ and $E$ are effective $\mathbb Q$-divisors on $X'$ such that 
$\Theta$ and $E$ have no common irreducible components. 
Then 
$$
K_{X'}+\Theta\sim _{\mathbb Q} f'^*(K_{Y'}+\Sigma _{Y'}+M_{Y'})
+E\sim _{\mathbb Q, f'}E\geq 0. 
$$
We run the minimal model program with respect to $K_{X'}+\Theta$ over $Y'$ (cf.~\cite[Proof of 
Theorem 1.1]{fujino-gongyo3}). 
Note that $(X', \Theta)$ is a $\mathbb Q$-factorial log canonical pair. 
Then we obtain a minimal model $\widetilde f: (\widetilde X, \widetilde \Theta)\to Y'$ such that 
$$
K_{\widetilde X}+\widetilde \Theta\sim _{\mathbb Q, \widetilde f}0. 
$$
It is easy to see that  
$$
K_{\widetilde X}+\widetilde \Theta\sim _{\mathbb Q}\widetilde f^*(K_{Y'}+\Sigma_{Y'}+M_{Y'}), 
$$ 
that is, $\Sigma_{Y'}$ is the discriminant $\mathbb Q$-divisor of $\widetilde f:(\widetilde X, \widetilde \Theta)\to Y'$ and 
$M_{Y'}$ is the moduli part of $\widetilde f:(\widetilde X, \widetilde \Theta)\to Y'$. 
Therefore, if we assume that the minimal model program and the abundance conjecture hold, 
then we can replace $(X, \Delta)$ with 
a log canonical pair $(\widetilde X, \widetilde \Theta)$ 
when we prove the b-semi-ampleness of $\mathbf M$ under the assumption that 
$\kappa (X_{\eta}, K_{X_{\eta}}+\Delta^+|_{X_{\eta}})=0$. 
We note that the b-semi-ampleness conjecture of $\mathbf M$ for $\widetilde f: (\widetilde X, \widetilde \Theta)\to Y'$ 
can be reduced to the case when $g:(V, \Delta_V)\to W$ is an lc-trivial fibration such that 
$(V, \Delta_V)$ is kawamata log terminal over 
the generic point of $W$ and $\Delta_V$ is effective. 
For the details, see \cite[Proof of Theorem 1.1]{fujino-gongyo3}.  

We also note that the existence of a good minimal model of $(X'_{\eta}, \Theta|_{X'_\eta})$, 
where $X'_\eta$ is the generic fiber of $f':X'\to Y'$, is sufficient 
to construct a relative good minimal model 
$\widetilde f:(\widetilde X, \widetilde \Theta)\to Y'$. 
Let us go into details. 
By replacing $(X', \Theta)$ with its dlt blow-up, we may assume that 
$(X', \Theta)$ is a $\mathbb Q$-factorial divisorial log terminal pair. 
We run the minimal model program on $K_{X'}+\Theta$ with ample scaling over $Y'$. 
After finitely many steps, all the horizontal components of $E$ are contracted if 
$(X'_\eta, \Theta|_{X'_{\eta}})$ has a good minimal model. 
Thus we assume that $E$ has no horizontal components. 
Then it is easy to see that $E$ is very exceptional over $Y'$. 
For the definition of {\em{very exceptional divisors}}, see, for example, 
\cite[Definition 3.1]{birkar}. 
Therefore, by \cite[Theorem 3.4]{birkar}, we obtain a relative 
minimal model $\widetilde f: (\widetilde X, \widetilde \Theta)\to Y'$ with 
$K_{\widetilde X}+\widetilde \Theta\sim _{\mathbb Q, \widetilde f}0$. 
Note that the existence of a good minimal model of $(X'_\eta, \Theta|_{X'_\eta})$ is equivalent to 
the existence of a good minimal model of 
$(X_\eta, \Delta^+|_{X_\eta})$. 
\end{say}

\begin{say}[Observation III]\label{215} 
Let $f:(X, \Delta)\to Y$ be an lc-trivial 
fibration such that $X$ and $Y$ are quasi-projective and that 
$(X, \Delta)$ is log canonical. 
By taking a dlt blow-up, we may assume that $(X, \Delta)$ is a $\mathbb Q$-factorial divisorial log terminal 
pair. 
Let $\overline Y$ be a normal projective variety which is a compactification of $Y$. 
By using the minimal model program, we can construct a projective $\mathbb Q$-factorial 
divisorial log terminal pair $(\overline X, \overline \Delta)$ which is a compactification of 
$(X, \Delta)$ such that $\overline X\setminus X$ contains no log canonical centers of 
$(\overline X, \overline \Delta)$ and that $f:X\to Y$ is extended to $\overline f: \overline X\to \overline Y$. 
$$
\xymatrix{
   (\overline X, \overline \Delta) \ar[d]_{\overline f} &\ar@{_{(}->}[l](X, \Delta)\ar[d]^f\\
   \overline Y  & \ar@{_{(}->}[l] Y
} 
$$
By \cite[Theorem 1.4]{birkar}, we have a good minimal model $(\overline X', \overline \Delta')$ over $\overline Y$.  
See also \cite[Theorem 1.1 and Corollary 1.2]{hacon-xu}. 
Let $\overline f': \overline X'\to \overline Y'$ be the contraction morphism over $\overline Y$ associated to 
$K_{\overline X'}+\overline \Delta'$. Then $\overline f': (\overline X', \overline \Delta')\to \overline Y'$ is an lc-trivial 
fibration which is a compactification of $f:(X, \Delta)\to Y$. 

Therefore, the b-semi-ampleness of $\mathbf M$ of $\overline f': (\overline X', \overline \Delta')\to \overline Y'$ 
implies that the moduli part of $f:(X, \Delta)\to Y$ is b-semi-ample over $S$, where $Y\to S$ is 
a proper morphism as in Conjecture \ref{conj29}. 
\end{say} 

By combining the above observations with the results in \cite[Theorem 3.3]{ambro3} 
and 
\cite[Theorem 8.1]{prokhorov-shokurov}, we have: 

\begin{thm}\label{thm216} 
Let $f:(X, \Delta)\to Y$ be an lc-trivial fibration such that 
$X$ is smooth and $\Supp \Delta$ is a simple normal crossing divisor 
on $X$ and let $Y\to S$ be a proper morphism. 
We write $\Delta=\Delta^+-\Delta^-$ where 
$\Delta^+$ and $\Delta^-$ are effective $\mathbb Q$-divisors and 
have no common irreducible components. 
Assume that 
$
\kappa (X_\eta, K_{X_\eta}+\Delta^+|_{X_\eta})=0 
$ 
where $X_\eta$ is the generic fiber of $f$ and 
that  $(X_\eta, \Delta^+|_{X_\eta})$ has a good minimal model. 
Then the moduli part $\mathbf M$ of $f:(X, \Delta)\to Y$ is b-nef and 
abundant over $S$. This means that there is a proper 
birational morphism $Y'\to Y$ from a normal variety $Y'$ such that 
$\mathbf M=\overline {\mathbf M_{Y'}}$ and $\mathbf M_{Y'}$ is 
nef and abundant relative to the induced morphism 
$Y'\to S$. 

We further assume that 
$\dim X=\dim Y+1$. 
Then the moduli part $\mathbf M$ of $f:(X, \Delta)\to Y$ is b-semi-ample 
over $S$. 
\end{thm}

We note that we do not use Theorem \ref{thm216} in the subsequent sections. 

\begin{proof}[Sketch of Proof of Theorem \ref{thm216}] 
By the arguments in \ref{214}, we may assume that 
$X$ and $Y$ are quasi-projective and that  
$(X, \Delta)$ is log canonical. 
By the arguments in \ref{215}, we may further assume that 
$X$ and $Y$ are projective. 
Then, by \cite[Theorem 1.1]{fujino-gongyo3}, we obtain that 
$\mathbf M$ is b-nef and abundant over $S$. 
When $\dim X=\dim Y+1$, we see that $\mathbf M$ is b-semi-ample over 
$S$ 
by \cite[Theorem 8.1]{prokhorov-shokurov}. 
\end{proof}

For the details of lc-trivial fibrations, see also \cite{ambro2} 
and \cite[Section 3]{fujino-gongyo3}. 

\section{Proof of the main results}\label{sec2} 
 
First, let us prove the log canonicity of $Y$ in Case \ref{case3} in the introduction 
by using 
Theorem \ref{thm11}. 

\begin{proof}[Proof of the log canonicity of $Y$ in 
Case \ref{case3}] 
It is easy to see that $Y$ is $\mathbb Q$-factorial (see, 
for example, \cite[Proposition 3.36]{kollar-mori}). 
By perturbing $\Delta$, we may assume that 
$\Delta$ is a $\mathbb Q$-divisor. 
By shrinking $Y$, we may assume that $Y$ is affine. 
We can take an effective $\mathbb Q$-divisor 
$\Delta'$ on $X$ such that 
$(X, \Delta+\Delta')$ is log canonical and that 
$$
K_X+\Delta+\Delta'\sim _{\mathbb Q, f}0. 
$$ 
Let $g:Z\to X$ be a resolution such that 
$$
K_Z+\Delta_Z=g^*(K_X+\Delta+\Delta') 
$$ 
and that $\Supp \Delta_Z$ is a simple normal crossing divisor on $Z$. 
Then 
$$
g_*\mathcal O_Z(\lceil -\Delta^{<1}_Z\rceil)\simeq \mathcal O_X. 
$$ 
Therefore, 
$$
h_*\mathcal O_Z(\lceil -\Delta^{<1}_Z\rceil)\simeq \mathcal O_Y
$$ 
and 
$$
K_Z+\Delta_Z\sim _{\mathbb Q, h}0
$$ 
where $h=f\circ g$. By Theorem \ref{thm11}, 
$Y$ has only log canonical singularities. 

If every log canonical center of $(X, \Delta)$ is dominant onto $Y$, 
then we can take $\Delta'$ such that every log canonical center of $(Z, \Delta_Z)$ 
is dominant onto $Y$. 
Thus $Y$ is log terminal by 
Theorem \ref{thm11} when every log canonical center of $(X, \Delta)$ is dominant 
onto $Y$. 
\end{proof}

Let us prove Theorem \ref{thm11}. We use the framework of lc-trivial fibrations.

\begin{proof}[Proof of Theorem \ref{thm11}] 
Without loss of generality, we may assume that $Y$ is affine. 
We can write 
$$
K_X+\Delta\sim _{\mathbb Q}f^*(K_Y+B_Y+M_Y)
$$ 
where $B_Y$ is the discriminant and 
$M_Y$ is the moduli part of the lc-trivial fibration $f:(X, \Delta)\to Y$ (see 
\ref{say28}). 
Note that $B_Y$ is effective (see, 
for example, the proof of \cite[Theorem 1.2]{fujino1}) 
and the coefficients of $B_Y$ are $\leq 1$.  
Let $E$ be an arbitrary prime divisor over $Y$. 
We take a resolution $\sigma:Y'\to Y$ with 
$$
K_{Y'}+B_{Y'}+M_{Y'}=\sigma^*(K_Y+B_Y+M_Y)
$$ 
such that $E$ is a prime divisor on $Y'$ and 
that $E\cup \Supp B_{Y'}\cup \Exc(\sigma)$ is a simple 
normal crossing divisor on $Y'$. 
Note that $f':(X', \Delta')\to Y'$ is 
an induced lc-trivial fibration by $\sigma:Y'\to Y$ (see \ref{say23}) 
and that $B_{Y'}$ is the discriminant and $M_{Y'}$ is the moduli part of 
$f': (X', \Delta')\to Y'$.  
$$
\xymatrix{
(X', \Delta_{X'})\ar[r]\ar[d]_{f'}& (X, \Delta)\ar[d]^{f} \\
 Y'\ar[r]_{\sigma} & Y  
} 
$$
By taking $\sigma:Y'\to Y$ suitably, we may assume that $M_{Y'}$ is 
$\sigma$-nef (see Theorem \ref{thm25}) and there is 
an effective exceptional $\mathbb Q$-divisor $F$ on $Y'$ 
which is anti-$\sigma$-ample. 
Without loss of generality, we may assume that 
the coefficients of $F$ are $\leq 1$. 
Let $\varepsilon$ be an arbitrary positive rational number. 
Then 
\begin{align*}
K_{Y'}+B_{Y'}+M_{Y'}&=K_{Y'}+B_{Y'}+\varepsilon F+M_{Y'}-\varepsilon F\\
&\sim _{\mathbb Q}K_{Y'}+B_{Y'}+\varepsilon F+G
\end{align*}
where $G$ is a general effective $\mathbb Q$-divisor 
on $Y'$ such that $\lfloor G\rfloor=0$, $G\sim _{\mathbb Q}M_{Y'}-\varepsilon F$, 
and 
$\Supp B_{Y'}\cup \Supp F\cup \Supp G$ is a simple normal crossing 
divisor. Note that $M_{Y'}-\varepsilon F$ is $\sigma$-ample and that 
$Y$ is affine. 
We put 
$$
\Theta_{E, \varepsilon}=\sigma_*(B_{Y'}+\varepsilon F+G). 
$$ 
Then $\Theta_{E, \varepsilon}$ is an effective $\mathbb Q$-divisor 
on $Y$ whose coefficients are 
$\leq 1$ such that $K_Y+\Theta_{E, \varepsilon}$ is $\mathbb Q$-Cartier 
and 
\begin{align*}
a(E, Y, \Theta_{E, \varepsilon})&=-\mult _E B_{Y'}-\varepsilon \mult _E F\\
&\geq -1-\varepsilon. 
\end{align*}
Therefore, 
$$
a(E, Y, 0)\geq a(E, Y, \Theta_{E, \varepsilon})\geq -1-\varepsilon. 
$$ 
This means that $a(E, Y, 0)\geq -1$. 
Thus $Y$ has only log canonical singularities. 

When every log canonical center of $(X, \Delta)$ is dominant onto $Y$, $\mult 
_E B_{Y'}<1$ 
always holds by the construction of $B_{Y'}$. 
Therefore, we obtain $a(E, Y, 0)>-1$. 
Thus $Y$ has only log terminal singularities. 
\end{proof}

\begin{rem}\label{rem-new} 
In the proof of Theorem \ref{thm11}, if $M_{Y'}$ is $\sigma$-semi-ample and 
$Y$ is quasi-projective, then we can take a general effective $\mathbb Q$-divisor 
$G$ on $Y'$ such that 
$$
K_{Y'}+B_{Y'}+M_{Y'}\sim _{\mathbb Q, \sigma}K_{Y'}+B_{Y'}+G. 
$$ 
Thus $(Y, \Delta_Y)$ is log canonical where 
$\Delta_Y=\sigma_*(B_{Y'}+G)$. 
Therefore, the b-semi-ampleness of $\mathbf M$ is desirable (see Conjecture 
\ref{conj29}). 
Of course, if $M_{Y'}$ is semi-ample, then we can choose $G$ such that 
$$
K_{Y'}+B_{Y'}+M_{Y'}\sim _{\mathbb Q} K_{Y'}+B_{Y'}+G. 
$$
\end{rem}

Note that $Y$ in Theorem \ref{thm11} has a quasi-log structure 
in the sense of Ambro (see \cite{ambro1}). 

\begin{rem}[Quasi-log structure]
We use the same notation as in Theorem \ref{thm11}. 
We can write 
$$
K_X+\Delta\sim _{\mathbb Q} f^*\omega
$$ 
for some $\mathbb Q$-Cartier $\mathbb Q$-divisor $\omega$ on $Y$. 
Then the pair $[Y, \omega]$ has a quasi-log structure with only qlc singularities 
(see \cite{ambro1}, \cite{fujino3}, \cite{fujino4}, and \cite{fujino-foundation}). 
Therefore, the cone and contraction theorem holds for $Y$ with respect to $\omega$. 
It is a complete generalization of \cite[Theorem 4.1]{fujino1}. 
\end{rem}

\begin{proof}[Proof of Theorem \ref{thm-add}] 
By using Theorem \ref{thm25}, 
the proof of Theorem 1.2 in \cite{fujino1} works. 
We leave the details as an exercise for the reader. 
\end{proof}

Let us start the proof of Theorem \ref{thm13}. 

\begin{proof}[Proof of Theorem \ref{thm13}] 
By taking a suitable resolution, we may assume that 
$f:X\to Y$ is a morphism 
such that 
$$
m_0(K_X+\Delta)=f^*A+E
$$
where $m_0$ is a sufficiently large and divisible positive integer, 
$A$ is a very ample Cartier divisor 
on $Y$, and $E$ is an effective 
divisor on $X$ satisfying 
$$
|mm_0(K_X+\Delta)|=|mf^*A|+mE
$$ 
for every positive integer $m$ (see, for example, \cite[Lemma 3.2]{birkar-div}). 
Without loss of 
generality, we may further assume that 
$\Supp \Delta\cup \Supp E$ is a simple normal crossing 
divisor on $X$. 
We put 
$$
\Delta_X=\Delta-\frac{1}{m_0}E. 
$$ 
Then we have 
$$
K_X+\Delta_X\sim _{\mathbb Q, f}0. 
$$
It is easy to see that 
$f_*\mathcal O_X(\lceil -\Delta^{<1}_X\rceil)\simeq \mathcal O_Y$ 
(see, for example, the proof of \cite[Lemma 3.2]{birkar}). 
Note that 
$$
0\leq \lceil -\Delta^{<1}_X\rceil =\lceil \frac{1}{m_0}E\rceil \leq E. 
$$ 
Therefore, by Theorem \ref{thm11}, 
we have that $Y$ has only log canonical singularities. 
\end{proof}

\begin{rem}
In the proof of Theorem \ref{thm13}, we have $\kappa (X_\eta, K_{X_\eta}+\Delta|_{X_\eta})=0$ where 
$X_\eta$ is the generic fiber of $f:X\to Y$. 
Therefore, if Conjecture \ref{conj29} holds under the extra assumption that 
$$\kappa(X_\eta, K_{X_\eta}+\Delta^+_X|_{X_\eta})=\kappa (X_\eta, K_{X_\eta}+\Delta|_{X_\eta})=0$$ as in 
\ref{213}, then Conjecture \ref{conj12} also holds (see Proof of Theorem \ref{thm11} and Remark \ref{rem-new}). 
\end{rem}

\begin{rem}
If $(X, \Delta)$ has a good minimal model in Conjecture \ref{conj12}, 
then we may assume that 
there is a morphism $f:X\to Y$ such that $f_*\mathcal O_X\simeq \mathcal O_Y$ 
and $K_X+\Delta\sim _{\mathbb Q, f}0$ by replacing 
$(X, \Delta)$ with its good minimal model. 
In this case, Conjecture \ref{conj12} follows from Conjecture \ref{conj15}. 
\end{rem}

\begin{proof}[Proof of Theorem \ref{thm14}]
By combining the proof of Theorem \ref{thm13} with Theorem \ref{thm-add}, 
we can find an effective $\mathbb Q$-divisor $\Delta_Y$ on $Y$ such that 
$(Y, \Delta_Y)$ is kawamata log terminal. 
We leave the details as an exercise for the reader.  
\end{proof}

We give a remark on the finite generation of $R(X, \Delta)$. 

\begin{rem}[Finite generation of $R(X, \Delta)$]
Let $(X, \Delta)$ be a projective log canonical pair such that 
$\Delta$ is a $\mathbb Q$-divisor. 
It is conjectured that the log canonical ring $R(X, \Delta)$ is 
a finitely generated $\mathbb C$-algebra. It is one of the most 
important conjectures for higher-dimensional algebraic varieties. 
For the details and various related conjectures, see \cite{fujino-gongyo4}. 
It is known that $R(X, \Delta)$ is finitely generated for $\dim X\leq 4$ 
(see \cite[Theorem 1.2]{fujino-finite}). 
Note that $R(X, \Delta)$ is a finitely generated $\mathbb C$-algebra 
when $(X, \Delta)$ is kawamata log terminal and $\Delta$ is a 
$\mathbb Q$-divisor. It was established by 
Birkar--Cascini--Hacon--M\textsuperscript{c}Kernan (\cite{bchm}) 
and is now well known. 
\end{rem}

We close this section with remarks on Conjecture \ref{conj15}. 

\begin{rem} 
If $(X, \Delta)$ is kawamata log terminal and $\Delta$ is a $\mathbb Q$-divisor 
in Conjecture \ref{conj15}, 
then we can take a $\mathbb Q$-divisor $\Delta_Y$ such that 
$(Y, \Delta_Y)$ is kawamata log terminal and $K_X+\Delta\sim _{\mathbb Q}f^*(K_Y+\Delta_Y)$ 
by \cite[Theorem 0.2]{ambro2}, 
which is a complete solution of \cite[Problem 1.1]{fujino1}. 
Theorem 3.1 in \cite{fujino-gongyo1} generalizes \cite[Theorem 0.2]{ambro2} 
for $\mathbb R$-divisors. 
\end{rem}

\begin{rem}[cf.~the proof of Theorem 3.1 in \cite{fujino-gongyo1}]\label{rem24}
In Conjecture \ref{conj15}, we can write 
$$
K_X+\Delta=\sum _{i=1}^kr_i(K_X+\Delta_i)
$$ 
such that 
\begin{itemize}
\item[(a)] $\Delta_i$ is an effective $\mathbb Q$-divisor 
for every $i$, 
\item[(b)] $(X, \Delta_i)$ is log canonical and $K_X+\Delta_i$ is $f$-nef 
for every $i$, and 
\item[(c)] $0<r_i<1$, $r_i\in \mathbb R$ for 
every $i$, and $\sum _{i=1}^k r_i=1$. 
\end{itemize}
Since $K_X+\Delta$ is numerically $f$-trivial, so is $K_X+\Delta_i$ for 
every $i$. 
By \cite[Theorem 4.9]{fujino-gongyo2}, 
we obtain that 
$K_X+\Delta_i\sim _{\mathbb Q, f}0$ 
for every $i$. 
Therefore, we can reduce Conjecture \ref{conj15} to the case 
when $\Delta$ is a $\mathbb Q$-divisor 
with $K_X+\Delta\sim _{\mathbb Q, f}0$. 
Then we can use the framework of lc-trivial fibrations. 
We can easily check that Conjecture \ref{conj15} follows from 
Conjecture \ref{conj29} such that $S$ is a point (see also Remark \ref{rem-new}). 
\end{rem}

\section{Some analytic generalizations}\label{sec3}

In this section, we give some remarks on complex analytic varieties in Fujiki's class $\mathcal C$ and compact K\"ahler manifolds. 
The following theorem easily follows from \cite{bchm} and \cite{fujino-mori}. 
Note that Theorem \ref{thm31} is equivalent to Theorem \ref{thm18} 
by taking a resolution. 

\begin{thm}[{cf.~\cite{bchm} and \cite{fujino-mori}}]\label{thm31}  
Let $X$ be a normal complex analytic variety in Fujiki's class $\mathcal C$ and 
let $\Delta$ be an effective $\mathbb Q$-divisor 
on $X$ such that $(X, \Delta)$ is kawamata log terminal. 
Then the log canonical ring 
$$
R(X, \Delta)=\bigoplus _{m\geq 0}H^0(X, \mathcal O_X(\lfloor m(K_X+\Delta)\rfloor))
$$ 
is a finitely generated $\mathbb C$-algebra. 
\end{thm}

As a special case of Theorem \ref{thm31}, we have: 

\begin{cor}\label{cor32} 
Let $X$ be a compact K\"ahler manifold. 
Then the canonical ring 
$$
R(X)=\bigoplus _{m \geq 0} H^0(X, \omega_X^{\otimes m})
$$ 
is a finitely generated $\mathbb C$-algebra. 
\end{cor}

Theorem \ref{thm31} and Corollary \ref{cor32} do not hold for 
varieties which are not in Fujiki's class $\mathcal C$ (see 
Example \ref{ex-wilson} below). 

Note that the proof of Theorem \ref{thm31} is not related to the minimal 
model theory for compact K\"ahler manifolds. 
We do not discuss the minimal models for compact K\"ahler 
manifolds here (see \cite{chp}, \cite{horing-peternell}, and \cite{horing-peternell2}). 

\begin{say}[Ideas]\label{say-ideas}
Let $m$ be a large and divisible positive integer and 
let $$\Phi_{|m(K_X+\Delta)|}: X\dashrightarrow Y$$ be the Iitaka fibration. 
Then $Y$ is {\em{projective}} even when $X$ is only a complex analytic variety. 
By taking suitable resolutions, it is sufficient to 
treat the case where $\Phi:X\to Y$ is a proper surjective morphism 
from a compact K\"ahler manifold $X$ to a normal projective variety $Y$ with 
connected fibers. In this situation, the arguments 
in \cite{fujino-mori}, \cite{ambro2}, \cite[Section 3]{fujino-gongyo3}, and so 
on work with some minor modifications. 
This is because we can use the theory of variations of (mixed) 
$\mathbb R$-Hodge structure for 
$\Phi:X\to Y$. In general, the general fibers of $\Phi$ are not projective. 
They are only K\"ahler. Therefore, the natural polarization of the 
variation of (mixed) Hodge structure is defined only on $\mathbb R$. 

Anyway, by the arguments in \cite[Sections 4 and 5]{fujino-mori}, 
we can find an effective $\mathbb Q$-divisor $\Delta_Y$ on $Y$ such that 
the finite generation 
of $R(X, \Delta)$ is equivalent to the finite generation of $R(Y, \Delta_Y)$ where 
$(Y, \Delta_Y)$ is kawamata log terminal and $K_Y+\Delta_Y$ is big. 
Therefore, by \cite{bchm}, $R(X, \Delta)$ is finitely generated. 
\end{say}

\begin{rem}
Let $f:X\to Y$ be a surjective morphism from a compact K\"ahler manifold 
(or, more generally, a complex analytic variety in Fujiki's class $\mathcal C$) 
$X$ to a projective variety $Y$. 
In this setting, we can prove various fundamental 
results, for example, Koll\'ar type vanishing theorem, torsion-free theorem, 
weak positivity theorem, and so on. 
For the details, see \cite{fujino7}. 
\end{rem}

By the arguments in \cite[Sections 4 and 5]{ambro2} and the semipositivity theorem 
in \cite{fujino2} (see also \cite{fujino-fujisawa}, \cite{ffs}, and 
\cite[Theorem 1.5]{fujino7}), we 
can prove an analytic generalization of Theorem \ref{thm25} 
without any difficulties (see Example \ref{ex-atiyah} and 
Remark \ref{rem63} for the case when the 
varieties are not in Fujiki's class 
$\mathcal C$). 

\begin{thm}[{cf.~\cite[Theorem 3.6]{fujino-gongyo3}}]\label{thm43} 
Let $X$ be a normal complex analytic variety in Fujiki's class $\mathcal C$ 
and let 
$\Delta$ be a $\mathbb Q$-divisor 
on $X$ such that 
$K_X+\Delta$ is $\mathbb Q$-Cartier. 
Let $f:X\to Y$ be a surjective 
morphism onto a normal projective variety $Y$. 
Assume that $f:(X, \Delta)\to Y$ is an lc-trivial fibration, 
that is, 
\begin{itemize}
\item[(i)] $(F, \Delta|_F)$ is sub log canonical for a general fiber $F$ of $f:X\to Y$, 
\item[(ii)] $\rank f_*\mathcal O_X(\lceil \mathbf A^*(X, \Delta)\rceil)=1$, 
and 
\item[(iii)] there exists a $\mathbb Q$-Cartier $\mathbb Q$-divisor 
$D$ on $Y$ such that 
$$
K_X+\Delta\sim _{\mathbb Q}f^*D. 
$$ 
\end{itemize}
Let $\mathbf B$ and $\mathbf M$ be the induced discriminant and 
moduli $\mathbb Q$-b-divisor of $f:(X, \Delta)\to Y$. 
Then 
\begin{itemize}
\item[(1)] $\mathbf K+\mathbf B$ is $\mathbb Q$-b-Cartier, that is, 
there exists a proper birational morphism 
$Y'\to Y$ from a normal variety $Y'$ such that $\mathbf K+\mathbf B=\overline 
{K_{Y'}+\mathbf B_{Y'}}$, 
\item[(2)] $\mathbf M$ is b-nef. 
\end{itemize}
\end{thm}

In the setting of Theorem \ref{thm43}, we have: 

\begin{conj}[{cf.~Conjecture \ref{conj29}}]\label{conj44} 
In Theorem \ref{thm43}, $\mathbf M$ is b-semi-ample. 
\end{conj}

Conjecture \ref{conj44} may be harder than Conjecture \ref{conj29} 
because there are no good moduli theory for compact K\"ahler manifolds. 

\begin{proof}[Proof of Theorem \ref{thm31}] 
Let $f:X\dashrightarrow Y$ be the Iitaka fibration with respect to 
$K_X+\Delta$. By replacing $X$ and $Y$ bimeromorphically, we may assume that 
$Y$ is a smooth projective variety, $X$ is a compact K\"ahler manifold, 
$(X, \Delta)$ is kawamata log terminal 
such that $\Supp \Delta$ is a simple normal crossing divisor on $X$, 
and $f$ is a morphism. 
By using the theory of log-canonical bundle formulas 
discussed in \cite[Section 4]{fujino-mori} with the aid of Theorem \ref{thm43}, 
we can apply \cite[Theorem 5.2]{fujino-mori}. 
Then there are a smooth projective variety $Y'$, which 
is birationally equivalent to $Y$, and an effective 
$\mathbb Q$-divisor $\Delta'$ on $Y'$ such that $(Y', \Delta')$ is a kawamata log 
terminal pair, $K_{Y'}+\Delta'$ is big, and 
$$
R(X, \Delta)^{(e)}\simeq R(Y', \Delta')^{(e')}
$$ 
for some positive integers $e$ and $e'$, 
where 
$$
R(Y', \Delta')=\bigoplus_{m\geq 0}H^0(Y', \mathcal O_{Y'}(\lfloor 
m(K_{Y'}+\Delta')\rfloor)).  
$$ 
Note that $R^{(e)}=\bigoplus _{m\geq 0}R_{em}$ for a graded ring 
$R=\bigoplus _{m\geq 0}R_m$. 
By \cite{bchm}, $R(Y', \Delta')$ is a finitely generated 
$\mathbb C$-algebra. 
This implies that $R(X, \Delta)$ is a finitely generated $\mathbb C$-algebra. 
\end{proof}

\begin{rem}
By using Theorem \ref{thm43}, we can prove some analytic generalizations 
of Theorem \ref{thm11}, Theorem \ref{thm-add}, and so on. 
We leave the details for the interested reader. 
\end{rem}

Conjecture \ref{conj45} is obviously equivalent to 
Conjecture \ref{conj20} by taking a resolution. 

\begin{conj}\label{conj45} 
Let $X$ be a normal complex analytic variety in Fujiki's class $\mathcal C$ 
and let $\Delta$ be an effective $\mathbb Q$-divisor 
on $X$ such that $(X, \Delta)$ is log canonical. 
Then 
the log canonical ring 
$$
R(X, \Delta)=\bigoplus _{m\geq 0}H^0(X, \mathcal O_X(\lfloor 
m(K_X+\Delta)\rfloor))
$$ 
is a finitely generated $\mathbb C$-algebra. 
\end{conj}

We give a remark on Conjecture \ref{conj45}. 

\begin{rem}\label{rem47} 
Conjecture \ref{conj45} is still open even when $X$ is projective. 
If Conjecture \ref{conj44} holds true, then we can reduce Conjecture 
\ref{conj45} to the case when $X$ is projective by the 
same way as in \cite[Sections 4 and 5]{fujino-mori} 
(see also \ref{say-ideas}, Proof of Theorem \ref{thm31}, and Remark \ref{rem-new}). 
Note that Theorem \ref{thm43} is not sufficient 
for this reduction argument. 
\end{rem}

We close this section with an observation on Conjecture \ref{conj44}. 

\begin{say}[Observation IV]\label{410}
Let $f:(X, \Delta)\to Y$ be an lc-trivial fibration as in Theorem \ref{thm43}. 
By taking a resolution, we assume that $X$ is a compact K\"ahler manifold and that $\Supp \Delta$ is 
a simple normal crossing divisor on $X$. 
For Conjecture \ref{conj44}, it seems to be reasonable 
to assume that  
$$
\rank f_*\mathcal O_X(\lceil m\Delta^-\rceil)=1
$$ 
for every nonnegative integer $m$ as in \ref{213}, equivalently, 
$$
\kappa (F, K_F+\Delta^+|_F)=0
$$ 
where $F$ is a sufficiently general fiber of $f$. 
The extra assumption $\kappa(F, K_F+\Delta^+|_F)=0$ is harmless for \cite[Sections 4 and 5]{fujino-mori}. 
Therefore, Remark \ref{rem47} works even if we add the extra assumption $\kappa(F, K_F+\Delta^+|_F)=0$ 
to Conjecture \ref{conj44}. 
Unfortunately, the reduction arguments 
in \ref{214} based on the minimal model program 
have not been established for compact K\"ahler manifolds. 

Anyway, Conjecture \ref{conj44} looks harder than Conjecture \ref{conj29}. 
\end{say}

\section{Examples of non-K\"ahler manifolds}\label{sec-counterexamples} 

In this section, we discuss some examples 
of compact complex non-K\"ahler manifolds 
constructed by Atiyah (\cite{atiyah}) and Wilson (\cite{wilson}) 
for the reader's convenience. These 
examples clarify the 
reason why we have to assume that 
the varieties are in Fujiki's class $\mathcal C$ in 
Section \ref{sec3}. 
For some related 
examples, see \cite[Remark 15.3]{ueno} and \cite{magnusson}. 

The following example is due to Atiyah (see \cite[\S 10, Specific examples]{atiyah}). 
This example shows that 
the Fujita--Kawamata semipositivity theorem does not hold 
for non-K\"ahler manifolds. 
For the details of the theory of 
fiber spaces of complex tori, see \cite{atiyah}. 

\begin{ex}[{cf.~\cite[\S 10]{atiyah}}]\label{ex-atiyah} 
Let us construct an analytic family of tori $f:X\to Y=\mathbb P^1$ such that 
$f_*\omega_{X/Y}$ is not semipositive. 

We put 
$$
I=\begin{pmatrix}0, &1\\ -1, &0
\end{pmatrix}, \quad
J=\begin{pmatrix}0, &\sqrt{-1}\\ \sqrt{-1},& 0
\end{pmatrix}, \quad
K=\begin{pmatrix}\sqrt{-1}, &0\\ 0, &-\sqrt{-1}
\end{pmatrix}, 
$$
and 
$$
E=\begin{pmatrix}1, &0\\
0, &1
\end{pmatrix}. 
$$
We take $s_1, s_2\in H^0(\mathbb P^1, \mathcal O_{\mathbb P^1}(1))\setminus \{0\}$ 
such that $s_1$ and $s_2$ have no common zeros. 
We consider the analytic family of tori 
$f:X\to Y:=\mathbb P^1$
where $$X=\mathbb V(\mathcal O_{\mathbb P^1}(1)\oplus 
\mathcal O_{\mathbb P^1}(1))/
\Lambda. $$
Note that $\mathbb V(\mathcal O_{\mathbb P^1}(1)\oplus 
\mathcal O_{\mathbb P^1}(1))=\mathrm{Spec}_{\mathbb P^1}
\mathrm{Sym}((\mathcal 
O_{\mathbb P^1}(1)\oplus \mathcal O_{\mathbb P^1}(1))^*)$ is the total space of 
$\mathcal O_{\mathbb P^1}(1)\oplus 
\mathcal O_{\mathbb P^1}(1)$ and 
$$
\Lambda=\left< E\begin{pmatrix}s_1\\ s_2\end{pmatrix}, \  
I\begin{pmatrix}s_1\\ s_2\end{pmatrix}, \ 
J\begin{pmatrix}s_1\\ s_2\end{pmatrix}, \ 
K\begin{pmatrix}s_1\\ s_2\end{pmatrix}
\right>.  
$$
In other words, the fiber $X_p=f^{-1}(p)$ for $p\in Y$ 
is $\mathbb C^2/\Lambda(p)$, 
where $\Lambda(p)$ is the lattice 
$$
\left<  
E\begin{pmatrix}s_1(p)\\ s_2(p)\end{pmatrix}, \ 
I\begin{pmatrix}s_1(p)\\ s_2(p)\end{pmatrix}, \ 
J\begin{pmatrix}s_1(p)\\ s_2(p)\end{pmatrix}, \ 
K\begin{pmatrix}s_1(p)\\ s_2(p)\end{pmatrix}
\right>_{\mathbb Z}
$$
in $\mathbb C^2$. For the 
details of the construction, see \cite{atiyah}. 
Then $\mathcal \omega_{X/Y}\simeq 
f^*\mathcal O_{\mathbb P^1}(-2)$ by \cite[Proposition 10]{atiyah}. 
Therefore, we have  
$$
f_*\mathcal \omega_{X/Y}\simeq \mathcal O_{\mathbb P^1}(-2). 
$$ 
This means that $f_*\omega_{X/Y}$ is not always 
semipositive when $X$ is not K\"ahler. 
Note that $f$ is smooth in this example. 
\end{ex}

\begin{rem} 
Let $f:V\to W$ be a surjective morphism 
from a compact complex manifold $V$ in Fujiki's class $\mathcal C$ to 
a smooth projective curve $W$. 
Then we can easily check that $f_*\omega_{V/W}$ is semipositive 
by \cite{fujita}. 
This means that $X$ in Example \ref{ex-atiyah} is not 
in Fujiki's class $\mathcal C$.  
\end{rem}

\begin{rem}\label{rem63} 
Example \ref{ex-atiyah} shows that 
Theorem \ref{thm43} does not hold without assuming that 
$X$ is in Fujiki's class $\mathcal C$. 
Therefore, the proof of Theorem \ref{thm31} 
does not work 
for varieties which are not in Fujiki's class $\mathcal C$. 
\end{rem}

The following example is essentially the same as Wilson's example 
(see \cite[Example 4.3]{wilson}). 
It is a compact complex non-K\"ahler $4$-fold whose canonical 
ring is not a finitely generated $\mathbb C$-algebra. 
Wilson's example is very important. 
Unfortunately, \cite[Example 4.3]{wilson} 
omitted some technical details. 
Moreover, we can not find it in the standard literature for the 
minimal model program. 
So we explain a slightly simplified example in details 
for the reader's convenience. 

\begin{ex}[{cf.~\cite[Example 4.3]{wilson}}]\label{ex-wilson} 
Let us construct a $4$-dimensional compact complex non-K\"ahler 
manifold $X$ whose canonical ring $R(X)$ is not a finitely 
generated $\mathbb C$-algebra. 

Let $C\subset \mathbb P^2$ be a smooth elliptic curve 
and let $H$ be a line on $\mathbb P^2$. 
We blow up 
$12$ general points $P_1, \cdots, P_{12}$ on $C$ 
and one point $P\not\in C$. 
Let $\pi:Z\to \mathbb P^2$ denote this birational modification 
and let $E$ be the exceptional curve $\pi^{-1}(P)$. 
Let $C'$ be the strict transform of $C$. 
We put $H'=\pi^*H-E$. 
Then the linear system $|H'|$ is free and $(H')^2=0$. 
Note that $K_{Z}\sim -C'+E$. 
\begin{cla}\label{claim1}
The linear system $|n\pi^*H+(n-1)C'|$ is free for every $n\geq 1$ and 
the base locus $\Bs|n\pi^*H+nC'|=C'$ for every $n\geq 1$. 
Therefore, we have 
$$
|n\pi^*H+nC'|=|n\pi^*H+(n-1)C'|+C' 
$$ for every $n\geq 1$. 
\end{cla}
\begin{proof}[Proof of Claim \ref{claim1}] 
It is obvious that $|n\pi^*H|$ is free. 
We consider the following short exact sequence: 
\begin{align}\label{shiki1}
0&\to \mathcal O_Z(n\pi^*H+(r-1)C')\to \mathcal O_Z(n\pi^*H+rC')\tag{$\spadesuit$}
\\ &\to 
\mathcal O_{C'}(n\pi^*H+rC')\to 0\notag
\end{align}
for $1\leq r\leq n$. 
Note that $\deg \mathcal O_{C'}(n\pi^*H+rC')\geq 3$ for $1\leq r\leq n-1$. 
Therefore, $|\mathcal O_{C'}(n\pi^*H+rC')|$ is very ample 
for $1\leq r\leq n-1$ since $C'$ is an elliptic curve. 
On the other hand, 
\begin{align*}
n\pi^*H+(r-1)C'-K_Z&\sim n\pi^*H+rC'-E\\
&=(n-1)\pi^*H+rC'+H'
\end{align*} 
is nef and big for $1\leq r\leq n-1$. 
By the Kawamata--Viehweg vanishing theorem, we obtain 
$$
H^1(Z, \mathcal O_Z(n\pi^*H+(r-1)C'))=0
$$ 
for $1\leq r\leq n-1$. By using the long exact sequence associated to 
\eqref{shiki1}, we have that $|n\pi^*H+rC'|$ is free for $1\leq r \leq n-1$ 
by induction 
on $r$. Note that 
$$
H^0(C', \mathcal O_{C'}(n\pi^*H+nC'))=0
$$ 
for every $n\ne 0$ since $P_1, \cdots, P_{12}$ are general points on $C$. 
Precisely speaking, we take $P_1, \cdots, P_{12}$ 
such that $\mathcal O_{C'}(\pi^*H+C')$ is not a torsion 
element in $\mathrm{Pic}^0(C')$.  
This means that the natural inclusion 
$$
0\to H^0(Z, \mathcal O_Z(n\pi^*H+(n-1)C'))\to H^0(Z, \mathcal O_Z(n\pi^*H+nC'))
$$ 
is an isomorphism for every $n\geq 1$. 
Thus we have $$|n\pi^*H+nC'|=|n\pi^*H+(n-1)C'|+C'$$ 
for every $n\geq 1$.  
\end{proof}

Similarly, we can check the following statement. 

\begin{cla}\label{claim2}
The linear system $|4\pi^*H+4H'+6C'-2E|$ is free. 
\end{cla}
\begin{proof}[Proof of Claim \ref{claim2}] 
We note that 
$$
4\pi^*H+4H'+6C'-2E=6H'+2\pi^*H+6C'. 
$$ 
We also note that $|H'|$ and $|\pi^*H|$ are free. 
We consider the linear system $|6H'+2\pi^*H+rC'|$ for 
$0\leq r\leq 6$. 
If $r=0$, then the linear system $|6H'+2\pi^*H|$ is free. 
We consider the following short exact sequence:  
\begin{align}\label{shiki2}
0&\to \mathcal O_Z(6H'+2\pi^*H+(r-1)C')\to 
\mathcal O_Z(6H'+2\pi^*H+rC')\tag{$\clubsuit$}
\\&\to \mathcal O_{C'}
(6H'+2\pi^*H+rC')\to 0. \notag
\end{align}
Note that 
\begin{align*}
6H'+2\pi^*H+(r-1)C'-K_Z&\sim 6H'+2\pi^*H-E+rC'\\
&=7H'+\pi^*H+rC'
\end{align*}
is nef and big for $1\leq r\leq 6$. 
Therefore, by the Kawamata--Viehweg vanishing theorem, 
we obtain $$
H^1(Z, \mathcal O_Z(6H'+2\pi^*H+(r-1)C'))=0
$$ 
for $1\leq r\leq 6$. On the other hand, 
$$
\deg \mathcal O_{C'}(6H'+2\pi^*H+rC')\geq 6
$$ 
for $1\leq r\leq 6$. 
Thus $|\mathcal O_{C'}(6H'+2\pi^*H+rC')|$ is very ample 
for $1\leq r\leq 6$. Note that $C'$ is an elliptic curve. 
By considering the long exact sequence associated to 
\eqref{shiki2} and by induction on $r$, 
we obtain that $|6H'+2\pi^*H+rC'|$ is free for $0\leq r\leq 6$. 
In particular, $|4\pi^*H+4H'+6C'-2E|$ is free. 
\end{proof}
We take a general member $C_0$ of the free linear system 
$|4\pi^*H+4H'+6C'-2E|$ and take the 
double cover $g:Y\to Z$ ramified along $C_0$. 
Then we have 
\begin{align*}
K_Y&=g^*(K_Z+2\pi^*H+2H'+3C'-E)\\
&\sim g^*(2\pi^*H+2H'+2C'). 
\end{align*}
Note that $|g^*H'|$ is free on a smooth projective 
surface $Y$ such that $\kappa(Y, g^*H')=1$. 
Then we can take $s_1, s_2\in H^0(Y, \mathcal O_Y(g^*H'))\setminus \{0\}$ such that 
$s_1$ and $s_2$ have no common zeros. 
By using $s_1$ and $s_2$, we can construct the analytic family of tori 
$f:X\to Y$ as in Example \ref{ex-atiyah}, that is, 
$$
X=\mathbb V(\mathcal O_Y(g^*H')\oplus \mathcal O_Y(g^*H'))/\Lambda
$$ 
and 
$$
\Lambda=\left< E\begin{pmatrix}s_1\\ s_2\end{pmatrix}, \  
I\begin{pmatrix}s_1\\ s_2\end{pmatrix}, \ 
J\begin{pmatrix}s_1\\ s_2\end{pmatrix}, \ 
K\begin{pmatrix}s_1\\ s_2\end{pmatrix}
\right>.  
$$
Then $X$ is a compact complex $4$-fold. 
By \cite[Proposition 10]{atiyah}, we can 
check that 
\begin{align*}
\omega_X&=f^*\mathcal O_Y(K_Y-2g^*H')\\
&=f^*\mathcal O_Y(g^*(2\pi^*H+2H'+2C')-2g^*H')\\&\simeq 
f^*\mathcal O_Y(g^*(2\pi^*H+2C')). 
\end{align*}
Therefore, if the canonical ring $R(X)=\bigoplus _{m\geq 0}
H^0(X, \omega_X^{\otimes m})$ is a finitely generated $\mathbb C$-algebra, 
then so is 
$$R(Z, 2\pi^*H+2C')=\bigoplus _{m\geq 0}H^0(Z, \mathcal O_Z(2m(\pi^*H+C'))).$$ 
\begin{cla}\label{claim3} 
$R(Z, 2\pi^*H+2C')=\bigoplus _{m\geq 0}H^0(Z, \mathcal O_Z(2m(\pi^*H+C')))$ 
is not a finitely generated $\mathbb C$-algebra. 
\end{cla}
\begin{proof}[Proof of Claim \ref{claim3}] 
By Claim \ref{claim1}, 
\begin{align*}
\bigoplus _{a>0}^{m-1}
&H^0(Z, \mathcal O_Z(2a(\pi^*H+C')))
\otimes H^0(Z, \mathcal O_Z(2(m-a)(\pi^*H+C')))
\\ &\to H^0(Z, \mathcal O_Z(2m(\pi^*H+C')))
\end{align*}
is not surjective for 
any $m\geq 1$. This implies that $R(Z, 2\pi^*H+2C')$ is not a finitely generated 
$\mathbb C$-algebra. 
\end{proof}
\begin{proof}[Alternative proof of Claim \ref{claim3}] 
Since $2\pi^*H+2C'$ is nef and big, we know that 
$R(Z, 2\pi^*H+2C')$ is a finitely generated 
$\mathbb C$-algebra if and only if $2\pi^*H+2C'$ is semi-ample (see, for example, 
\cite[Theorem 2.3.15]{lazarsfeld}). 
On the other hand, $\mathcal O_{C'}(\pi^*H+C')$ is not a torsion element 
in $\Pic^0(C')$. 
This implies that $\pi^*H+C'$ is not semi-ample. 
Therefore, $R(Z, 2\pi^*H+2C')$ is not a finitely generated $\mathbb C$-algebra.
\end{proof}
Therefore, the canonical ring $R(X)$ of $X$ is not finitely generated 
as a $\mathbb C$-algebra. 
Since $f_*\omega_{X/Y}\simeq \mathcal O_Y(-2g^*H')$ is not 
nef, $X$ is non-K\"ahler. 
Note that $X$ is a compact complex manifold which is not in Fujiki's 
class $\mathcal C$ by Theorem \ref{thm31}. 
\end{ex}

Example \ref{ex-wilson} shows that 
the finite generation of canonical rings 
does not always hold for compact complex manifolds which are not in 
Fujiki's class $\mathcal C$. 

\begin{rem}
Wilson's original example (see \cite[Example 4.3]{wilson}) 
uses the fact that $nH'+(n-1)C'$ is very ample for all $n\geq 1$ 
(see \cite[Claim in Example 4.3]{wilson}). 
Since $H'$ is not a big divisor, the statement in 
\cite[Claim in Example 4.3]{wilson} has to be changed suitably. 
So, we modified his construction slightly. 
Note that $f:X\to Y$ constructed in Example \ref{ex-wilson} 
does not coincide with Wilson's original 
example $V\to \widetilde S$ in \cite[Example 4.3]{wilson}.  
\end{rem}

Example \ref{ex-wilson} also shows that 
there are no generalizations of the abundance conjecture for 
compact complex non-K\"ahler manifolds. 

\begin{rem}
In Example \ref{ex-wilson}, we can check that $\pi^*H+C'$ is nef 
and big. Therefore, $\omega_X$ is a pull-back of 
a nef and big line bundle on a smooth projective variety $Y$. 
So $X$ should be recognized to be a minimal model. 
However, $\omega_X$ is not semi-ample. 
This means that 
the abundance conjecture can not be generalized 
for compact complex manifolds which are not 
in Fujiki's class $\mathcal C$. 
\end{rem}

We close this section with a comment on Moriwaki's result (see \cite{moriwaki}). 

\begin{rem}
Let $X$ be a three-dimensional compact complex manifold. 
Moriwaki proved that 
the canonical ring $R(X)$ of $X$ is always a finitely generated 
$\mathbb C$-algebra even when $X$ is not K\"ahler 
(see \cite[(3.5) Theorem]{moriwaki}). 
\end{rem}

\section{Appendix}\label{sec-appendix}

In this appendix, we quickly discuss the minimal model program for 
log canonical pairs and describe some 
related examples by J\'anos Koll\'ar 
for the reader's convenience. 
We assume that all the varieties are algebraic throughout 
this section. 

\subsection{Minimal model program for log canonical pairs}

Let $\pi:(X, \Delta)\to S$ be a projective morphism such that 
$(X, \Delta)$ is a $\mathbb Q$-factorial log canonical pair. 
Then we can run the minimal model 
program on $(X, \Delta)$ over $S$ since we have the cone and contraction 
theorem (see, for example, \cite[Theorem 1.1]{fujino6}) 
and the flip theorem for log canonical pairs (see \cite[Corollary 1.2]{birkar} 
and \cite[Corollary 1.8]{hacon-xu}). 
We can also run the minimal model program on $(X, \Delta)$ over $S$ with 
scaling by \cite[Theorem 1.1]{fujino6}. 
Unfortunately, we do not know if the minimal model program 
terminates or not. 

\begin{conj}[Flip conjecture II]\label{conj61} 
A sequence of flips 
$$
(X_0, \Delta_0)\dashrightarrow 
(X_1, \Delta_1)\dashrightarrow (X_2, \Delta_2)\dashrightarrow 
\cdots
$$
terminates after finitely many steps. Namely, 
there exists no infinite sequence of flips. 
\end{conj}

Note that each flip in Conjecture \ref{conj61} 
is a flip described in Case \ref{case2} in the introduction. 

If Conjecture \ref{conj61} is true, then we can freely 
use the minimal model program in full generality.  
In order to prove Conjecture \ref{conj61} in dimension $n$, 
it is sufficient to solve Conjecture \ref{conj61} for kawamata 
log terminal pairs in dimension $\leq n$. 
This reduction is an easy consequence of the existence of 
dlt blow-ups and 
the special termination theorem by induction on the dimension. 
For the details, see \cite{fujino-special} and \cite{fujino-foundation}. 

More generally, by the cone and contraction theorem (see \cite[Theorem 1.1]{fujino6}),  
\cite[Theorem 1.1]{birkar} and \cite[Theorem 1.6]{hacon-xu}, 
we can run the minimal model 
program for non-$\mathbb Q$-factorial 
log canonical pairs (see \cite[Subsection 3.1.2]{fujino3} and 
\cite{fujino-foundation}). 
Note that the termination of flips in this 
more general setting also follows from Conjecture \ref{conj61} for 
kawamata log terminal pairs 
by using the existence of dlt blow-ups and the special termination theorem 
as explained above (see \cite{fujino-special} and \cite{fujino-foundation}). 

Anyway, Conjecture \ref{conj61} for $\mathbb Q$-factorial 
kawamata log terminal pairs 
is one of the most important open problems for the minimal 
model program. 

\subsection{On log canonical flops} 

In this subsection, we discuss some examples, which 
show the differences between kawamata log terminal pairs and 
log canonical pairs. The following result is well known to the 
experts (see, for example, \cite[Theorem 3.24]{fujino3}). 

\begin{thm}\label{thm61}
Let $(X, \Delta)$ be a kawamata log terminal pair and let $D$ be 
a $\mathbb Q$-divisor on $X$. 
Then $\bigoplus _{m\geq 0}\mathcal O_X(\lfloor mD\rfloor)$ is a finitely generated 
$\mathcal O_X$-algebra. 
\end{thm}

\begin{proof}
If $D$ is $\mathbb Q$-Cartier, then this theorem is obvious. 
So we assume that $D$ is not $\mathbb Q$-Cartier. 
Since the statement is local, we may assume that 
$X$ is affine. By replacing $D$ with $D'$ such that 
$D'\sim D$ and $D'\geq 0$, we may further assume that $D$ is effective. 
By \cite{bchm}, we can take a small projective birational morphism $f:Y\to X$ such that 
$Y$ is $\mathbb Q$-factorial, $K_Y+\Delta_Y=f^*(K_X+\Delta)$, and 
$(Y, \Delta_Y)$ is kawamata log terminal. 
Let $D_Y$ be the strict transform of $D$ on $Y$. 
Note that $D_Y$ is $\mathbb Q$-Cartier because $Y$ is $\mathbb Q$-factorial. 
Let $\varepsilon$ be a small positive number. 
By running the minimal model program on $(Y, \Delta_Y+\varepsilon D_Y)$ 
over $X$ with scaling, we may assume that $D_Y$ is $f$-nef. 
Then, by the basepoint-free theorem, 
$D_Y$ is $f$-semi-ample. Therefore, 
$$
\bigoplus _{m\geq 0}f_*\mathcal O_Y(\lfloor mD_Y\rfloor)
$$ 
is a finitely generated $\mathcal O_X$-algebra. 
Since we have an $\mathcal O_X$-algebra isomorphism 
$$
\bigoplus _{m\geq 0}f_*\mathcal O_Y(\lfloor mD_Y\rfloor)\simeq \bigoplus 
_{m\geq 0}\mathcal O_X(\lfloor mD\rfloor), 
$$ 
$\bigoplus _{m\geq 0}\mathcal O_X(\lfloor mD\rfloor)$ is a finitely generated 
$\mathcal O_X$-algebra. 
\end{proof}

The next example shows that 
Theorem \ref{thm61} does not always hold for log canonical pairs. 
In other words, 
if $(X, \Delta)$ is log canonical, then $\bigoplus _{m\geq 0}
\mathcal O_X(\lfloor mD\rfloor)$ 
is not necessarily finitely generated as 
an $\mathcal O_X$-algebra. 

\begin{ex}[{\cite[Exercise 95]{kollar3}}]\label{exe88} 
Let $E\subset \mathbb P^2$ be a smooth cubic 
curve. 
Let $S$ be a surface obtained by blowing up nine 
sufficiently general points on $E$ and let $E_S\subset S$ be 
the strict transform of $E$. Let $H$ be a very 
ample divisor on $S$ giving 
a projectively normal embedding $S\subset \mathbb P^N$. 
Let $X\subset \mathbb A^{N+1}$ be the cone 
over $S$ and let $D\subset X$ be the cone over $E_S$. 
Then $(X, D)$ is log canonical by Lemma \ref{lem63} below 
since $K_S+E_S\sim 0$. 
Let $P\in D\subset X$ be the vertex 
of the cones $D$ and $X$. 
Since $X$ is normal, we have 
\begin{align*}
H^0(X, \mathcal O_X(mD)) &=H^0(X\setminus P, \mathcal O_X(mD))
\\& \simeq \bigoplus _{r\in \mathbb Z}H^0(S, 
\mathcal O_S(mE_S+rH)). 
\end{align*} 
By construction, 
$\mathcal O_S(mE_S)$ has only the obvious section which 
vanishes 
along $mE_S$ for every $m>0$. 
It can be checked by induction on $m$ 
using the following exact sequence 
\begin{align*}
0&\to H^0(X, \mathcal O_S((m-1)E_S))\to 
H^0(S, \mathcal O_S(mE_S))\\ &\to H^0(E_S, \mathcal 
O_{E_S}(mE_S))\to \cdots 
\end{align*}
since $\mathcal O_{E_S}(E_S)$ is not a torsion element 
in $\Pic ^0(E_S)$. 
Therefore, 
$$H^0(S, \mathcal O_S(mE_S+rH))=0$$ for every $r<0$. So, we 
have 
$$
\bigoplus _{m\geq 0} \mathcal O_X(mD)
\simeq \bigoplus _{m\geq 0}\bigoplus _{r\geq 0} 
H^0(S, \mathcal O_S(mE_S+rH)). 
$$ 
Since $E_S$ is nef, $\mathcal O_S(mE_S+4H)\simeq 
\mathcal O_S(K_S+E_S+mE_S+4H)$ is very 
ample for every $m\geq 0$. 
Therefore, by replacing $H$ with 
$4H$, we may assume that 
$\mathcal O_S(mE_S+rH)$ is very ample for 
every $m\geq 0$ and every $r>0$. 
In this setting, the multiplication maps 
\begin{align*}
\bigoplus _{a=0}^{m-1}H^0(S, \mathcal O_S(aE_S+H))\otimes 
H^0(S, \mathcal O_S((m-a)E_S))
\\ \to H^0(S, \mathcal O_S(mE_S+H))
\end{align*} 
are never surjective.  
This implies that 
$\bigoplus _{m\geq 0}\mathcal O_X(mD)$ 
is not finitely generated as an 
$\mathcal O_X$-algebra. 
\end{ex}

Let us recall an easy lemma for the reader's convenience. 

\begin{lem}\label{lem63} 
Let $(V, \Delta)$ be a log canonical pair such that $V$ is smooth, 
$\Supp \Delta$ is a simple normal crossing divisor on $V$, and 
$K_V+\Delta\sim _{\mathbb Q}0$. 
Let $V\subset \mathbb P^N$ be a projectively normal embedding. 
Let $W\subset \mathbb A^{N+1}$ be the cone over 
$V$ and let $\Delta_W$ be the cone over $\Delta$. 
Then $(W, \Delta_W)$ is log canonical. 
\end{lem}
\begin{proof}
Let $g:W'\to W$ be the blow-up at $0\in \mathbb A^{N+1}$ 
and let $E$ be the exceptional divisor of $g$.  
Note that $W'$ is smooth and $E\simeq V$. 
Then we can check that 
$$
K_{W'}+\Delta_{W'}+E=g^*(K_W+\Delta_W)
$$ 
where $\Delta_{W'}$ is the strict transform of $\Delta_W$. 
Note that $\Supp (\Delta_{W'}+E)$ is a simple normal crossing 
divisor on $W'$. Thus, $(W, \Delta_W)$ is log canonical. 
\end{proof}

Let us recall the definition of log canonical flops.  

\begin{defn}[Log canonical flop]\index{log canonical 
flop}\label{lc-flop-def}
Let $(X, \Delta)$ be a log canonical pair. Let $D$ be a 
Cartier divisor on $X$. 
Let $f:X\to Y$ be 
a small contraction such that 
$K_X+\Delta$ is numerically $f$-trivial and 
$-D$ is $f$-ample. 
The opposite of $f$ with respect to $D$ is called 
a {\em{flop with respect to $D$ for  
$(X, \Delta)$}} or simply a {\em{$D$-flop}}. 
We sometimes call it 
{\em{flop}} or a {\em{log canonical flop}} if there is no risk 
of confusion. 
\end{defn} 

\begin{rem}\label{rem66} 
Without loss of generality, we may assume that 
$\Delta$ is a $\mathbb Q$-divisor and $K_X+\Delta\sim _{\mathbb Q, f}0$ 
in Definition \ref{lc-flop-def} (see, for example, 
Remark \ref{rem24} and \cite[Theorem 4.9 and Subsection 4.1]{fujino-gongyo2}). 
Furthermore, if $(X, \Delta+\varepsilon D)$ is log canonical 
for some positive number $\varepsilon$, then a $D$-flop 
always exists by \cite[Theorem 1.1 and 
Corollary 1.2]{birkar} and \cite[Theorem 1.6 and Corollary 1.8]{hacon-xu}. 
\end{rem}

The following example shows that 
log canonical flops do not always exist. 
Of course, flops always exist for kawamata log terminal pairs by \cite{bchm}. 

\begin{ex}[{\cite[Exercise 96]{kollar3}}]\label{exe89} 
Let $E$ be an elliptic curve and let $L$ be a degree 
zero line bundle on $E$. 
We put $$S=\mathbb P_E(\mathcal O_E\oplus L).$$ Let 
$C_1$ and $C_2$ be the sections of the $\mathbb P^1$-bundle 
$p:S\to E$. 
We note that $K_S+C_1+C_2\sim 0$. 
As in Example \ref{exe88}, 
we take a sufficiently ample divisor 
$H=aF+bC_1$ on $S$ giving a projectively normal embedding 
$S\subset 
\mathbb P^N$, where $F$ is a fiber of the $\mathbb P^1$-bundle 
$p:S\to E$, $a>0$, and $b>0$. 
We may assume that $\mathcal O_S(mC_i+rH)$ is very ample 
for $i=1, 2$, every $m\geq 0$, and every $r>0$. 
Moreover, we may assume that 
$\mathcal O_S(M+rH)$ is 
very ample for any nef Cartier divisor $M$ and 
every $r>0$. 
Let $X\subset \mathbb A^{N+1}$ be the cone over $S$ and 
let $D_i\subset X$ be the cones over 
$C_i$ for $i=1$ and $2$. 
Since $K_S+C_1+C_2\sim 0$, $(X, D_1+D_2)$ is log canonical 
by Lemma \ref{lem63}. 
We can check $K_X+D_1+D_2\sim 0$ by construction. 
By the same arguments as in Example \ref{exe88}, 
we can prove the following statement. 
\setcounter{cla}{0}
\begin{cla}\label{exe-cl1}
If $L$ is a non-torsion element in 
$\Pic ^0(E)$, then 
$$\bigoplus _{m\geq 0} \mathcal O_X(mD_i)$$ is not 
a finitely generated $\mathcal O_X$-algebra 
for $i=1$ and $2$. 
\end{cla}
We note that $\mathcal O_S(mC_i)$ has only the obvious 
section which vanishes along $mC_i$ for every $m>0$. 

Let $B\subset X$ be the cone over $F$. 
Then we have the following result. 
\begin{cla}\label{exe-cl2} 
The graded $\mathcal O_X$-algebra $\bigoplus _{m\geq 0} 
\mathcal O_X(mB)$ is a finitely generated 
$\mathcal O_X$-algebra. 
\end{cla} 
\begin{proof}[Proof of Claim \ref{exe-cl2}]
By the same arguments as in Example \ref{exe88}, 
we have 
$$
\bigoplus _{m\geq 0} \mathcal O_X(mB)\simeq 
\bigoplus _{m\geq 0} 
\bigoplus _{r\geq 0} H^0(S, \mathcal O_S(mF+rH)). 
$$ 
We consider $V=\mathbb P_S(\mathcal O_S(F)\oplus \mathcal 
O_S(H))$. 
Then $\mathcal O_V(1)$ is semi-ample. 
Therefore, 
$$
\bigoplus _{n\geq 0} H^0(V, \mathcal O_V(n))\simeq 
\bigoplus _{m\geq 0} \bigoplus _{r\geq 0} 
H^0(S, \mathcal O_S(mF+rH))  
$$ 
is finitely generated. 
\end{proof}
Let $P \in X$ be the vertex of the cone $X$ and 
let $f:Y\to X$ be the blow-up at $P$. 
Let $A\simeq S$ be the exceptional divisor 
of $f$. 
We consider the $\mathbb P^1$-bundle $\pi:
\mathbb P_S(\mathcal O_S\oplus \mathcal O_S(H))\to 
S$. Then 
$$Y\simeq \mathbb P_S(\mathcal O_S\oplus \mathcal O_S(H))\setminus G,$$  
where $G$ is the section of 
$\pi$ corresponding to 
$$\mathcal O_S\oplus \mathcal O_S(H)
\to \mathcal O_S(H)\to 0.$$  
We consider $\pi^*F$ on $Y$. 
Then $\mathcal O_Y(\pi^*F)$ is obviously 
$f$-semi-ample. 
So, we obtain a contraction morphism 
$g:Y\to Z$ over $X$. 
We can check that 
$$Z\simeq \Proj _X\bigoplus _{m \geq 0} \mathcal O_X(mB)$$ 
over $X$ and that $h:Z\to X$ is a small projective contraction. 
On $Y$, we have 
$$-A\sim \pi^*H=a\pi^*F+b\pi^*C_1. $$  
Therefore, we obtain $aB+bD_1\sim 0$ on $X$. 
Let $B'$ be the strict transform of $B$ on $Z$ and let $D'_i$ be the strict transform 
of $D_i$ on $Z$ for $i=1$ and $2$. 
Note that $B'$ is $h$-ample, $aB'+bD'_1\sim 0$, and 
$K_Z+D'_1+D'_2=h^*(K_X+D_1+D_2)\sim 0$. 
If $L$ is not a torsion element, then 
the flop of $h:Z\to X$ with respect to 
$D'_1$ for $(Z, D'_1+D'_2)$ does not 
exist since $\bigoplus _{m\geq 0} \mathcal O_X(mD_1)$ is 
not finitely generated as an $\mathcal O_X$-algebra. 

Let $C$ be any Cartier divisor on $Z$ such that 
$-C$ is $h$-ample. 
Then the flop of $h:Z\to X$ with respect to 
$C$ exists if and only if 
$$\bigoplus _{m\geq 0} h_*\mathcal O_Z(mC)$$ is a finitely 
generated $\mathcal O_X$-algebra. 
We can take positive integers $m_0$ and $m_1$ such that 
$m_1C$ is numerically equivalent to 
$m_0D'_1$ over $X$. 
Note that $\Exc(h)\simeq E$. 
Therefore, we can find a degree zero Cartier divisor 
$N$ on $E$ such that 
$$
m_1C-m_0D'_1\sim _{h}g_*(\pi|_Y)^*(p^*N). 
$$ 
Thus, 
$$
\bigoplus _{m\geq 0}h_*\mathcal O_Z(mm_1C)
$$ 
is a finitely generated $\mathcal O_X$-algebra if and only 
if 
$$
R=\bigoplus _{m\geq 0}h_*\mathcal O_Z(m(m_0D'_1+g_*(\pi|_Y)^*
(p^*N)))
$$ 
is so. 
Since $h$ is small, 
$R$ is isomorphic to 
$$
\bigoplus _{m\geq 0}\mathcal O_X(m(m_0D_1+\widetilde N)), 
$$ 
where $\widetilde N\subset X$ is the cone over $p^*N$. 
Anyway, 
$$
\bigoplus _{m\geq 0}h_*\mathcal O_Z(mC)
$$ 
is a finitely generated $\mathcal O_X$-algebra 
if and only if 
$$
\bigoplus _{m\geq 0}\mathcal O_X(m(m_0D_1+\widetilde N))
$$ 
is so, where $\widetilde N$ is the cone over $p^*N$. 
We note the following commutative diagram: 
$$
\xymatrix{
&Z\ar[dl]_{h}&\\
X\ar@{-->}[rrd]&&Y\ar[ll]^{f}\ar[d]^{\pi|_Y}\ar[ul]_{g}\\
&&S\ar[d]^{p}\\
&&E
}
$$
where $X\dashrightarrow S$ is the natural projection from the vertex $P$ of 
$X$. 

\begin{cla}\label{exe-cl3}
If $L$ is not a torsion element in $\Pic ^0(E)$, 
then $$\bigoplus_{m\geq 0} 
\mathcal O_X(m(m_0D_1+\widetilde N))$$ 
is not finitely generated as an $\mathcal O_X$-algebra. 
In particular, the flop of $h:Z\to X$ with 
respect to $C$ does not exist. 
\end{cla}
\begin{proof}[Proof of Claim \ref{exe-cl3}] 
By the same arguments as in Example \ref{exe88}, 
we have 
\begin{align*}
&\bigoplus_{m\geq 0} \mathcal O_X(m(m_0D_1+\widetilde N))
\\
&\simeq  
\bigoplus _{m\geq 0}\bigoplus _{r\in \mathbb Z}
H^0(S, \mathcal O_S(m(m_0C_1+p^*N)+rH)). 
\end{align*} 
By considering 
\begin{align*}
0&\to H^0(S, \mathcal O_S((l-1)C_1+mp^*N))\to 
H^0(S, \mathcal O_S(lC_1+mp^*N))\\ 
&\to H^0(C_1, \mathcal O_{C_1}(lC_1+mp^*N))\to \cdots
\end{align*} 
for $1\leq l\leq mm_0$, we obtain that 
$$\dim H^0(S, \mathcal O_S(m(m_0C_1+p^*N)))\leq 1$$ for 
every $m\geq 0$. Therefore, we can check that the above 
$\mathcal O_X$-algebra is not finitely generated 
by the same arguments as in Example \ref{exe88}. 
We note that $\mathcal O_S(m(m_0C_1+p^*N)+rH)$ is very ample 
for every $m\geq 0$ and every $r>0$ because 
$m_0C_1+p^*N$ is nef. 
\end{proof}
Anyway, if $L$ is not a torsion element in $\Pic ^0(E)$, 
then the flop of $h:Z\to X$ does not exist with respect to any divisor. 

From now on, in the above setting, we assume that $L$ is a torsion 
element in $\Pic ^0(E)$. 
Then $\mathcal O_Y(\pi^*C_1)$ is $f$-semi-ample. So, 
we obtain a contraction morphism 
$g':Y\to Z^+$ over $X$. 
It is easy to see that $$\bigoplus _{m\geq 0}
\mathcal O_X(mD_i)$$ is 
finitely generated as an $\mathcal O_X$-algebra 
for $i=1, 2$ (cf.~Claim \ref{exe-cl2}), 
$$Z^+\simeq \Proj _X\bigoplus _{m\geq 0}\mathcal O_X(mD_1), 
$$  
over $X$ and that $Z^+\to X$ is the flop of $Z\to X$ with 
respect to $D'_1$. 

Let $C$ be any Cartier divisor on $Z$ such that 
$-C$ is $h$-ample. 
If $-C\sim _{\mathbb Q, h} cB'$ 
for some positive rational number 
$c$, then it is obvious that the above $Z^+\to X$ is the flop of 
$h:Z\to X$ with respect to $C$. 
If $-C\not\sim _{\mathbb Q, h}cB'$ for any 
positive rational number $c$, 
then the flop of $h:Z\to X$ with respect to $C$ does not 
exist. As above, 
we take positive integers $m_0$ and $m_1$ such that 
$m_1C$ is numerically equivalent to 
$m_0D'_1$ over $X$. 
Then we can find 
a degree zero 
Cartier divisor $N$ on $E$ such that 
$$
m_1C-m_0D'_1\sim _{h}g_*(\pi|_Y)^*(p^*N). 
$$ 
Since $-C\not\sim _{\mathbb Q, h}cB'$ for any positive rational 
number $c$, 
$N$ is a non-torsion element in $\Pic^0(E)$. 
Thus, 
$$\bigoplus _{m\geq 0} h_*\mathcal O_Z(mC)$$ is 
finitely generated if and only if 
$$\bigoplus _{m\geq 0} \mathcal O_X(m(m_0D_1+\widetilde 
N))$$ is so, where $\widetilde N\subset X$ is the cone 
over $p^*N\subset S$. By the 
same arguments as in the proof of Claim \ref{exe-cl3}, 
we can check that $$\bigoplus _{m\geq 0}
\mathcal O_X(m(m_0D_1+\widetilde N))$$ is not finitely 
generated as an $\mathcal O_X$-algebra. 
We note that 
$$\dim H^0(S, \mathcal O_S(m(m_0C_1+p^*N)))=0$$ for every $m>0$ 
since $N$ is a non-torsion element in $\Pic ^0(E)$ and 
$L$ is a torsion element in $\Pic ^0(E)$ 
(see the proof of Claim \ref{exe-cl3}).   
\end{ex}


\end{document}